%% file: markov_estimate.tex
\documentclass[11pt,letterpaper,twoside,reqno,nosumlimits]{amsart}

\allowdisplaybreaks

\input{macro-3}
\usepackage[ruled]{algorithm2e}

\usepackage[numbers]{natbib}

\newcommand{\la}{\langle}
\newcommand{\ra}{\rangle}

\mathtoolsset{showonlyrefs}

\begin{document}

\title[non-asymptotic estimates for transition matrices]{Non-asymptotic Estimates for Markov Transition Matrices\\via Spectral gap methods}

\author[D. Huang, X. Li]{De Huang$^1$, Xiangyuan Li$^2$}
\thanks{$^1$School of Mathematical Sciences, Peking University. E-mail: dhuang@math.pku.edu.cn}
\thanks{$^2$School of Mathematical Sciences, Peking University. E-mail: lixiangyuan23@stu.pku.edu.cn}

\begin{abstract}
We establish non-asymptotic error bounds for the classical Maximal Likelihood Estimation of the transition matrix of a given Markov chain. Meanwhile, in the reversible case, we propose a new reversibility-preserving online Symmetric Counting Estimation of the transition matrix with non-asymptotic deviation bounds. Our analysis is based on a convergence study of certain Markov chains on the length-2 path spaces induced by the original Markov chain.
 \end{abstract}
 \maketitle

\section{Introduction}
Markov models are fundamental in various fields for their efficiency in modeling time-dependent scenarios where the future state depends solely on the present state. The inverse problem of estimating the transition matrix of a given Markov chain is one of the most fundamental tasks in many scientific applications~\cite{mostel2020statistical,craig2002estimation,inamura2006estimating,trendelkamp2015estimation,macdonald2021fitting}.

Let $\Omega$ be a finite state space, and let $\mtx{P}=[p(u,v)]_{u,v\in \Omega}$ be an irreducible transition matrix with an invariant distribution $\mu$. Our goal is to estimate the transition probability $p(u,v)$ for all pairs of states when $\mtx{P}$ is unknown. Suppose we have access to a finite Markov sequence $\{u_k\}_{k=1}^{n+1}$ generated by $\mtx{P}$. An extensively studied method, known as the Maximal Likelihood Estimation (MLE), simply counts the number of jumps from $u$ to $v$ for all $u,v\in \Omega$ and uses the empirical mean of the indicator function $\mathbb{1}_{(u,v)}(\cdot,\cdot)$ to approximate $\mu(u)p(u,v)$:   
\begin{equation}\label{eqt:formula of MLE}
    r_n(u,v):=\frac{1}{n}\sum\limits_{k=1}^{n}\mathbb{1}_{(u,v)}\left(u_k,u_{k+1}\right)\approx \mu(u)p(u,v).
\end{equation}
Then $\mu(u)$ and $p(u,v)$ can be approximated by $\sum_{w\in \Omega} r_n(u,w)$ and $ r_n(u,v)/\sum_{w\in \Omega} r_n(u,w) $, respectively. We will use $\mtx{R}_n=[r_n(u,v)]_{u,v\in \Omega}$ to denote the matrix form of MLE. It is shown in~\cite{anderson1957statistical,billingsley1961statistical} that the error of MLE is asymptotically $\mathcal{O}(1/\sqrt{n})$ by the central limiting theorem for Markov chains. In contrast, our goal is to establish non-asymptotic results that are valid for any $n\in\mathbb{N}_+$.

Our first contribution is to establish non-asymptotic error bounds of MLE  in terms of the iterated Poincar\'e gap (IP gap) of $\mtx{P}$, which was first introduced in \cite{chatterjee2025spectral} as a generalization of the spectral gap for non-reversible Markov chains. We call it the iterated Poincar\'e gap, following the notations in \cite{huang2026bernstein}, to emphasize its definition by an iterated Poincar\'e inequality (see Subsection \ref{subsec:Fundamentals} below). To achieve this, we interpret MLE as the Markov empirical mean of the indicator function $\mathbb{1}_{(u,v)}(\cdot,\cdot)$ corresponding to some transition matrix $\mtx{P}_2$ defined on the length-2 path space $\Omega^2$. We then establish the relation between the spectral gaps of $\mtx{P}_2$ and $\mtx{P}$, and apply concentration inequalities for Markov chains to derive the following entry-wise tail bound for the error of MLE.

\begin{theorem}\label{thm:MLEsimple version}
    For any $t>0$ and any $u,v\in\Omega$, 
    \[\mathbb{P}\left\{|r_n(u,v)-\mu(u)p(u,v)|\geq t\right\}\leq C_1\|\nu/\mu\|_{\infty}\exp\left(-C_2\eta_p(\mtx{P})nt^2\right),\]
    where $C_1$, $C_2$ are absolute constants, $\nu$ is the initial distribution of the Markov sequence, and $\eta_p(\mtx{P})$ is the IP gap of $\mtx{P}$.
\end{theorem}

There have been many related works that also focus on learning and testing Markov transition matrices \cite{hao2018learning,wolfer2021statistical,chan2021learning,wolfer2024empirical}. Most relevantly, \cite{wolfer2021statistical} controls the probabilistic sample complexity of MLE at the matrix level (in terms of the total variation norm). We will discuss their results in detail in Section~\ref{sec:main results}. We want to emphasize that, compared with prior work, the novelty of our result lies in that it only requires a nonzero IP gap of $\mtx{P}$ (which proves to be more robust than other generalized spectral gaps for non-reversible Markov chains; see \cite{chatterjee2025spectral,huang2026bernstein}), owing to a recently established Bernstein-type inequality \cite{huang2026bernstein} by the same authors. We will show that $\eta_p(\mtx{P})>0$ holds for any irreducible transition matrix $\mtx{P}$. Hence, Theorem~\ref{thm:MLEsimple version} applies to all Markov chains with irreducible transition matrices, highlighting the robustness of our result.

In numerous applications, the Markov chain considered is reversible~\cite{grone2008interlacing}, i.e. the transition matrix satisfies the detailed balance condition. In this case, it is important that the estimator preserves the reversibility of $\mtx{P}$ ~\cite{annis2010estimation,trendelkamp2015estimation,macdonald2021fitting}. However, the theories proposed in the literature do not provide any non-asymptotic convergence rate guarantee for those estimators.

Our second contribution is to propose an online reversibility-preserving Symmetric Counting Estimation (SCE) method as follows in the reversible case. We construct a reversible Markov chain $\{\tilde{\mtx{u}}_k\}_{k\geq 1}$ on the symmetrized length-2 path space $\tilde{\Omega}^2$ and use its empirical mean under some matrix-valued function $\tilde{\mtx{F}}$ to approximate $\mtx{D}_\mu\mtx{P}$:
\[ \mtx{H}_n:=\frac{1}{n}\sum\limits_{k=1}^n\tilde{\mtx{F}}(\tilde{\mtx{u}}_k)\approx \mtx{D}_\mu\mtx{P}.\]
Let $\tilde{\mtx{P}}_2$ denote the corresponding transition matrix on $\tilde{\Omega}^2$. A non-asymptotic error bound of SCE in the matrix form can be obtained by exploiting the spectral properties of $\tilde{\mtx{P}}_2$ and applying concentration inequalities for Markov dependent random matrices recently established in~\cite{neeman2024concentration}.
\begin{theorem}\label{thm:SCEsimple version}
    For any $t>0$, 
    \[\mathbb{P}\left\{\left\|\mtx{H}_n-\mtx{D}_{\mu}\mtx{P}\right\|\geq t\right\}\leq C_1d^{2-\pi/4}\|\nu/\mu\|_{\infty}\exp\left(-C_2\eta(\mtx{P})nt^2\right),\]
    where $d=|\Omega|$, $C_1$, $C_2$ are absolute constants, $\nu$ is some initial distribution on $\Omega$, and $\eta(\mtx{P})$ is the spectral gap of $\mtx{P}$.
\end{theorem}

Non-asymptotic expected error bounds are sometimes more favorable in applications. A simple way to bound the error expectation is by integrating the corresponding tail bound, which can be crude since the tail bound is in general not sharp for all scales of deviations. For example, if we integrate the Bernstein-type tail bounds above to control the expected errors for MLE and SCE, the resulting bounds will typically depend on the dimension $d$, which is in fact a technical overestimation. Instead, we derive dimension-free expectation bounds of both estimators via a direct calculation by utilizing the special structure $\mtx{P}_2$ and $\tilde{\mtx{P}}_2$.
\begin{theorem}\label{thm:dimension free error bound}
    It holds that
    \[\E\left[\left\|\mtx{R}_n-\mtx{D}_{\mu}\mtx{P}\right\|_{\mathrm{F}}\right]=\mathcal{O}\left(\frac{1}{\sqrt{n\eta_p(\mtx{P})}}\right),\]
    where $\eta_p(\mtx{P})$ is the IP gap of $\mtx{P}$. Moreover, when $\mtx{P}$ is reversible,
    \[\E\left[\left\|\mtx{R}_n-\mtx{D}_{\mu}\mtx{P}\right\|_{\mathrm{F}}\right]=\mathcal{O}\left(\frac{1}{\sqrt{n\eta(\mtx{P})}}\right),\quad \E\left[\left\|\mtx{H}_n-\mtx{D}_{\mu}\mtx{P}\right\|_{\mathrm{F}}\right]=\mathcal{O}\left(\frac{1}{\sqrt{n\eta(\mtx{P})}}\right).\]
\end{theorem} 
Our numerical results suggest that for each of the two methods, the mean square error (MSE) is independent of $d$, and that 1/MSE depends linearly on $n$ and the spectral gap, which justifies the optimality of our theoretical results.

The rest of this paper is organized as follows. In Section~\ref{sec:main results}, we develop some basic facts about Markov chains and state our main results which are complete versions of the theorems above. Section~\ref{sec:proofs} focuses on constructing the induced Markov chains on path spaces and exploring their properties. We then use these properties to prove our main results in Section~\ref{sec:error bounds}. The presentation of numerical results in Section~\ref{sec:numerical results} concludes the paper.

\section{Non-asymptotic estimates of Markov matrices}\label{sec:main results}
In this section, we outline the main results of this paper with proofs appearing in Section~\ref{sec:error bounds}, before which we develop some basic facts about Markov chains.
\subsection{Fundamentals}\label{subsec:Fundamentals}
\subsubsection{Notation}
Throughout this paper, matrices are written in uppercase bold letters. In particular, $\mtx{E}_{uv}$ denotes the matrix with $1$ at position $(u,v)$ and $0$ elsewhere, $\mtx{\Id}_d$ denotes the $d$-dimensional identity matrix, and $\mtx{1}_d$ denotes the d-dimensional all-ones vector. In cases where it does not cause ambiguity, we will omit the subscript $d$ to use $\bf{\Id}$ and $\bf{1}$ instead of $\Id_d$ and $\mtx{1}_d$.

For a real vector $\alpha\in\R^d$, we use $\|\alpha\|_{\infty}$ for the $L_{\infty}$ norm of $\alpha$ and $\mtx{D}_{\alpha}$ for the $d\times d$ diagonal matrix with the elements of $\alpha$ on the diagonal. For a matrix $\mtx{A}\in \R^{m\times n}$, we write $\|\mtx{A}\|$ for the $\ell_2$ operator norm, $\|\mtx{A}\|_{\mathrm{F}}$ for the Frobenius norm, $\mtx{A}(u,v)$ for the entry at position $(u,v)$, and $\mtx{A}^{\mathrm{T}}$ for the transpose. For $\mtx{A}\in \R^{d\times d}$, we write $\operatorname{spec}(\mtx{A})$ for the spectrum set and $\trace[\mtx{A}]$ for the trace.

Let $\mathbb{H}_d$ denote the linear space of all $d\times d$ real symmetric matrices and $\mathbb{H}_d^+$ denote the cone of all positive-semidefinite matrices. We use the symbol $\preceq$ for the semidefinite partial order on Hermitian matrices: for matrices $\mtx{A}, \mtx{B} \in \mathbb{H}_d$, the inequality $\mtx{A} \preceq \mtx{B}$ means that $\mtx{B}-\mtx{A} \in \mathbb{H}_d^{+}$. 

Let $\Omega$ be a finite state space equipped with a probability measure $\mu$, the number of states $|\Omega|$ is denoted by $d$. Any real-valued function on $\Omega$ can be seen as a finite-dimensional real vector and vice versa. Define $\mathbb{E}_\mu$ the expectation with respect to the measure $\mu$. When applied to a random matrix, $\mathbb{E}_{\mu}$ computes the entry-wise expectation. 
\subsubsection{Foundation of Markov chains}
Let $\{u_k\}_{k\geq 1}$ be a discrete-time Markov chain on a finite state space $\Omega$, and let $\mtx{P}=[p(u, v)]_{u,v \in \Omega}$ denote the corresponding transition matrix. Let $\mu$ be an invariant distribution of $\mtx{P}$ in the sense that
\begin{equation}\label{eqt:invariant distribution}
    \sum\limits_{u \in \Omega} \mu(u) p(u, v)=\mu(v), \quad \text { for all}\,\ v \in \Omega.
\end{equation}
Throughout this paper, we will assume that the transition matrix $\mtx{P}$ is irreducible, and thus its invariant distribution $\mu$ is unique. For a function $f:\Omega\rightarrow \mathbb{R}$, the action of $\mtx{P}$ as a linear operator on $f$ is defined as 
\[\mtx{P}f(u)=\mathbb{E}\left[f(u_2)|u_1=u\right]=\sum\limits_{v\in \Omega}p(u,v)f(v).\]
We say $\{u_k\}_{k\geq 1}$ is reversible with respect to $\mu$ if 
\[\mathbb{E}_{\mu}\left[g \cdot \mtx{P}f\right]=\mathbb{E}_{\mu}\left[f \cdot \mtx{P}g\right], \quad\text{for any}\,\ f,g:\Omega\rightarrow \R.\]
The reversibility of $\{u_k\}_{k\geq 1}$ is equivalent to the detailed balance condition of $\mtx{P}$:
\begin{equation}\label{eqt:detailed balance}
    \mu(u)p(u,v)=\mu(v)p(v,u),\quad \text{for any}\,\ u,v\in \Omega.
\end{equation}
In this case, we also say that $\mtx{P}$ is reversible.

Let $L_{2,\mu}$ be the Hilbert space of all real-valued functions on $\Omega$ endowed with the inner product 
\[\langle f,g \rangle_{\mu}=\mathbb{E}_{\mu}\left[f\cdot g\right],\]
and $L_{2,\mu}^0=\{f \in L_{2,\mu}: \E_\mu[f]=0\}$ be the mean-zero subspace of $L_{2,\mu}$. We will use $\|\cdot\|_{\mu}$ to denote both the $L_{2,\mu}$ norm of a vector and the $L_{2,\mu}$ operator norm of a matrix. 

In the reversible case, $\mtx{P}$ is self-adjoint under the $L_{2,\mu}$ inner product, and the \textit{spectral gap} of the transition matrix $\mtx{P}$ is given by  $\eta(\mtx{P}):=1-\lambda(\mtx{P})$, where
\[\lambda(\mtx{P}):=\sup _{h \in L_{2,\mu}^0,\ h \neq \mtx{0}} \frac{\la h,\mtx{P} h\ra_{\mu}}{\|h\|_{\mu}^2}.\]
In other words, $\eta(\mtx{P})$ is the gap between 1 and the second largest eigenvalue of $\mtx{P}$. 

When $\mtx{P}$ is non-reversible, we define the \textit{iterated Poincar\'e gap} (IP gap) of $\mtx{P}$ as 
\[\eta_p(\mtx{P}):=\inf _{h \in L_{2,\mu}^0,\ h \neq \mtx{0}} \frac{\|(\mtx{\Id}-\mtx{P}) h\|_{\mu}}{\|h\|_{\mu}}.\]
This definition coincides with that of the non-reversible spectral gap first formally introduced by Chatterjee in \cite{chatterjee2025spectral}, where he used this quantity to establish ergodicity for non-reversible Markov chains, generalizing similar results for reversible Markov chains based on the usual spectral gap. The reader may thus also refer to the IP gap as the Chatterjee gap or the non-reversible spectral gap and can find a more comprehensive introduction to the topic in \cite{chatterjee2025spectral}. This quantity was independently developed by \cite{huang2026bernstein} to establish Bernstein-type inequalities for Markov chains. In this paper, we choose to adopt the notation in \cite{huang2026bernstein} and call this quantity the iterated Poincar\'e gap, since it is related to the iterated Poincar\'e inequality:
\[\operatorname{Var}_\mu[h] \leq \eta_p(\mtx{P})^{-2} \mathbb{E}_\mu\left[((\mtx{\Id}-\mtx{P}) h)^2\right], \quad \text { for any } h \in L_{2, \mu} .\]

\subsection{Main results}
In this subsection, we outline the main results of this paper, which are complete versions of Theorem \ref{thm:MLEsimple version}, \ref{thm:SCEsimple version} and \ref{thm:dimension free error bound} in the introduction. Let $\{u_k\}_{k\geq 1}$ be a Markov chain on $\Omega$ and $\mtx{P}$ be the corresponding irreducible transition matrix we aim to estimate. The length-2 path space $\Omega^2$ is defined to be the Cartesian product of $\Omega$ with itself. The first contribution of the paper is to construct a Markov chain $\{\mtx{u}_k\}_{k\geq 1}$ on $\Omega^2$ induced by $\{u_k\}_{k\geq 1}$ (see Section \ref{sec:proofs}) and establish non-asymptotic error bounds of MLE by interpreting it as the empirical mean under some matrix-valued function $\mtx{F}$ defined on $\Omega^2$:
\[ \mtx{R}_n=\frac{1}{n}\sum\limits_{k=1}^n\mtx{F}(\mtx{u}_k).\]
 To be specific, in the following theorem we give entry-wise tail bounds and dimension-free expectation bounds for the error of MLE in terms of $n$ and the spectral gap of $\mtx{P}$.
\begin{theorem}\label{thm:main results of natural counting}
    Let $\{u_k\}_{k\geq 1}$ be a Markov chain on $\Omega$ generated by an irreducible transition matrix $\mtx{P}$ with an invariant distribution $\mu$, and let $\mtx{R}_n=[r_n(u,v)]_{u,v\in\Omega}$ be the matrix form of the corresponding MLE. 
    \begin{enumerate}
        \item \label{item:MLEmainresults convergence} $\mtx{R}_n$ converges to $\mtx{D}_\mu\mtx{P}$ almost surely.
        \item \label{item:MLEmainresults tail bound} Let $\nu$ be the initial distribution of $u_1$. For any $t>0$ and any $u,v\in \Omega$,
        \[\mathbb{P}\left\{\left|r_n(u,v)-\mu(u)p(u,v)\right|\geq t\right\}\leq 2\|\nu/\mu\|_{\infty}\exp\left(\frac{-n\, \gamma(\mtx{P})\, t^2}{4\sqrt{\left(2+6\gamma(\mtx{P})\right)^2\mu_2(u,v)+t^2}}\right),\]
        where $\mu_2(u,v)=\mu(u)p(u,v)$, $\gamma(\mtx{P})=\eta_p(\mtx{P})/(1+\eta_p(\mtx{P}))$.
        \item \label{item:MLEmainresults MSE}It holds that 
        \[\E\left[\left\|\mtx{R}_n-\mtx{D}_{\mu}\mtx{P}\right\|_{\mathrm{F}}^2\right]\leq \|\nu/\mu\|_{\infty}\min \left\{  \frac{4+\eta_p(\mtx{P})}{n\eta_p(\mtx{P})},\frac{2+\eta_a(\mtx{P})}{n\eta_a(\mtx{P})} \right\}.\]
        Moreover, when $\mtx{P}$ is reversible,
        \[\E\left[\left\|\mtx{R}_n-\mtx{D}_{\mu}\mtx{P}\right\|_{\mathrm{F}}^2\right]\leq \|\nu/\mu\|_{\infty}\frac{2+\eta(\mtx{P})}{n\eta(\mtx{P})}.\]
    \end{enumerate}
\end{theorem}

When $\mtx{P}$ is reversible, the Markov chain constructed in Theorem~\ref{thm:main results of natural counting} is non-reversible in general. In order to utilize the reversibility of $\mtx{P}$, we can symmetrize $\Omega^2$ by viewing $(u_1,u_2)$ and $(u_2,u_1)$ as the same element to define a symmetric length-2 path space $\tilde{\Omega}^2$ and construct a reversible transition matrix $\tilde{\mtx{P}}_2$ on $\tilde{\Omega}^2$. When we have access to an online oracle of $\mtx{P}$ that takes an input $u\in \Omega$ and outputs a state $v\in \Omega$ with probability $p(u,v)$, we propose a Symmetric Counting Estimation (SCE) algorithm that generates a sequence $\{\tilde{\mtx{u}}_k\}_{k\geq 1}\subset \tilde{\Omega}^2$ according to $\tilde{\mtx{P}}_2$ and approximates $\mtx{D}_\mu\mtx{P}$ using the symmetric matrix empirical mean 
\[\mtx{H}_n=\frac{1}{n}\sum\limits_{k=1}^n\tilde{\mtx{F}}(\tilde{\mtx{u}}_k),\]
where $\tilde{\mtx{F}}$ is some matrix-valued function defined on $\tilde{\Omega}^2$ to be specified later. We will show that $\tilde{\mtx{P}}_2$ admits a nonzero absolute spectral gap that is exactly half of the spectral gap of $\mtx{P}$. As a result, we can obtain non-asymptotic error bounds for SCE as in the following theorem. 

\begin{theorem}\label{thm:main results of symmetric counting}
    Let $\mtx{P}$ be an irreducible reversible transition matrix on $\Omega$ with an invariant distribution $\mu$, and let the corresponding symmetric path space $\tilde{\Omega}^2$ be defined as above. The followings hold:
    \begin{enumerate}
        \item \label{item:SCEmainresults convergence} There exists a reversible Markov chain $\{\tilde{\mtx{u}}_k\}_{k\geq 1}$ on $\tilde{\Omega}^2$ and a symmetric matrix $\mtx{H}_n=[h_n(u,v)]_{u,v\in\Omega}$ generated by $\{\tilde{\mtx{u}}_k\}_{k\geq 1}$ such that $\mtx{H}_n$ converges to $\mtx{D}_\mu\mtx{P}$ almost surely.
        \item \label{item:SCEmainresults tail bound}For any $t>0$,
        \[\mathbb{P}\left\{\left\|\mtx{H}_n-\mtx{D}_{\mu}\mtx{P}\right\|\geq t\right\}\leq 2\|\nu/\mu\|_{\infty}d^{2-\frac{\pi}{4}}\exp\left(\frac{-n\, \eta(\mtx{P})\,t^2}{128\pi^{-2} (2-\eta(\mtx{P})/2)+1024\pi^{-3}t}\right),\]
        where $\nu$ is some initial distribution on $\Omega$ which is related to the generation of $\tilde{\mtx{u}}_1$.
        \item \label{item:SCEmainresults MSE} It holds that
        \[\E\left[\left\|\mtx{H}_n-\mtx{D}_{\mu}\mtx{P}\right\|_{\mathrm{F}}^2\right]\leq \|\nu/\mu\|_{\infty}\frac{4-\eta(\mtx{P})}{n\eta(\mtx{P})}.\]
   \end{enumerate}
\end{theorem}
\subsection{Generalized spectral gaps}\label{subsec:other spectral gaps}
In the non-reversible case, various notions of spectral gaps have been proposed as substitutes for the usual spectral gap before \cite{chatterjee2025spectral}. A commonly used quantity is the \textit{absolute $L_2$-spectral gap} (or simply the \textit{absolute spectral gap}), which is defined as $\eta_a(\mtx{P})=1-\lambda_a(\mtx{P})$, where 
\[\lambda_a(\mtx{P}):=\sup _{h \in L_{2,\mu}^0,\ h \neq \mtx{0}} \frac{\|\mtx{P} h\|_{\mu}}{\|h\|_{\mu}}=\|\mtx{P}-\mtx{\Pi}\|_{\mu}.\]
Here $\mtx{\Pi}:=\mtx{1}\mu^{\mathrm{T}}$ is the projection onto $\mtx{1}$ under the $L_{2,\mu}$ inner product. One can see that $\eta_a(\mtx{P})$ is the gap between $1$ and the second largest singular value of $\mtx{P}$ under the $L_{2,\mu}$ inner product. Alternatively, $\eta_a(\mtx{P})$ can be understood as the usual spectral gap of the self-adjoint operator $\left(\mtx{P}^*\mtx{P}\right)^{1/2}$, where $\mtx{P}^*$ is the adjoint of $\mtx{P}$ under the $L_{2,\mu}$ inner product. This quantity was widely used in the literature to bound the mixing time \cite{fill1991eigenvalue} and establish concentration inequalities \cite{lezaud1998chernoff,leon2004optimal} for non-reversible Markov chains.

A generalization of the absolute spectral gap, called the \textit{pseudo spectral gap}, was first introduced in \cite{paulin2015concentration} to establish concentration inequalities for non-reversible Markov chains. It is defined as 
\[\eta_{ps}(\mtx{P}):=\sup_{k\in \N_+}\frac{1}{k}\left(1-\lambda_{a}^2\left(\mtx{P}^k\right)\right)=\sup_{k\in \N_+}\frac{\eta\left((\mtx{P}^*)^k\mtx{P}^k\right)}{k}.\] 
Advanced theoretical studies, estimations and applications of the pseudo spectral gap can be found in \cite{wolfer2019estimating,wolfer2024improved}.

Another way to define the spectral gap for a non-reversible Markov chain is by simply using the usual spectral gap of the self-adjoint operator $\mtx{A}:=(\mtx{P}^*+\mtx{P})/2$, the additive symmetrization of $\mtx{P}$. To be specific, the \textit{symmetric spectral gap} of $\mtx{P}$ is defined as $\eta_s(\mtx{P}):=1-\lambda_s(\mtx{P})$, where 
\[\lambda_s(\mtx{P}):= \sup _{h \in L_{2,\mu}^0,\ h \neq \mtx{0}} \frac{\la h,\mtx{P} h\ra_{\mu}}{\|h\|_{\mu}^2}=  \sup _{h \in L_{2,\mu}^0,\ h \neq \mtx{0}} \frac{\la h,\mtx{A} h\ra_{\mu}}{\|h\|_{\mu}^2}=\lambda(\mtx{A}),\]
 It is shown in \cite{fill1991eigenvalue} that $\eta_s(\mtx{P})$ can be used to get upper bounds on the mixing time of $\mtx{P}$ when $\mtx{P}$ is strongly aperiodic. Moreover, the additive reversiblization $\mtx{A}$ is usually more user-friendly than the multiplicative reversiblization $\mtx{P}^*\mtx{P}$. However, except for \cite{fan2021hoeffding, jiang2018bernstein}, which appear to contain a technical gap in the proof of a key lemma (see \cite{huang2026bernstein} for details), there are no concentration results that depend on $\eta_s(\mtx{P})$.

 Note that in the reversible case, it always holds that $\eta(\mtx{P})=\eta_p(\mtx{P})= \eta_s(\mtx{P})$, and if the second largest eigenvalue of $\mtx{P}$ in absolute value is positive, then it also holds that $\eta_p(\mtx{P}) =\eta_a(\mtx{P})$. When $\mtx{P}$ is non-reversible, one can prove that the relations $\eta_p(\mtx{P})\geq \eta_s(\mtx{P})\geq \eta_a(\mtx{P})$ and $\eta_p(\mtx{P})\geq \eta_{ps}(\mtx{P})/2$ (Lemma \ref{lem:comparison_of_spectral gap}) hold for any non-reversible Markov transition matrix. Moreover, by the Perron--Frobenius theorem, the irreducibility of $\mtx{P}$ (which implies the irreducibility of $(\mtx{P}^*+\mtx{P})/2$) guarantees that $\eta_p(\mtx{P})\geq \eta_s(\mtx{P})>0$. However, irreducibility does not guarantee that $\eta_a(\mtx{P})>0$. We can only show that $\lambda_a(\mtx{P})\leq 1$ (and hence $\eta_a(\mtx{P})\geq 0$) by Jensen's inequality:
\begin{equation}\label{lambda_a leq 1}
    \|\mtx{P} h\|_{\mu}^2=\sum\limits_{u\in \Omega}\mu(u)\left(\sum\limits_{v\in\Omega} p(u,v)h(v)\right)^2 \leq \sum\limits_{u\in \Omega}\mu(u)\sum\limits_{v\in\Omega}p(u,v) h(v)^2 = \|h\|_{\mu}^2.
\end{equation}
In fact, there are simple examples where $\eta_p(\mtx{P})>0$ but the ratios $\eta_a(\mtx{P})/\eta_p(\mtx{P})$ and $\eta_{ps}(\mtx{P})/\eta_p(\mtx{P})$ can be arbitrarily small or even zero (see Appendix \ref{sec:appendix}).

For technical convenience, in the rest of this paper we will assume $\mtx{P}$ is a dense matrix, i.e. $p(u,v)\ne 0$ for any $u,v\in\Omega$.  Note that all spectral properties discussed above are well-posed under small perturbation on $\mtx{P}$, and hence all arguments and conclusions in this paper can be readily generalized to non-dense cases by the limiting argument.
\subsection{Related works}
There have been many related works that also focus on learning and testing Markov transition matrices \cite{hao2018learning,wolfer2021statistical,chan2021learning,wolfer2024empirical}. Most relevantly, \cite{wolfer2021statistical} controls the probabilistic sample complexity of MLE in terms of the mixing time $t_{\text{mix}}$. More specifically, we say the transition matrix $\mtx{P}$ is geometrically ergodic, if there exists $\rho\in(0,1)$ such that for any $u\in \Omega$,
\[\left\|\mtx{P}^n(u,\cdot)-\mu^{\mathrm{T}}\right\|_{\text{TV}}\leq C \rho^n,\] where $\|\cdot\|_{\text{TV}}=\|\cdot\|_{\ell_1}/2$ is the standard total variation norm. The mixing time of $\mtx{P}$ is then defined by 
\[t_{\text{mix}}:=\inf\left\{n:\sup_{\nu} \left\|\nu^{\mathrm{T}}\mtx{P}^n-\mu^{\mathrm{T}} \right\|_{\text{TV}}\leq 1/4\right\},\]
where the supremum is taken over all distributions over $\Omega$.

It is shown in \cite[Theorem 3.3]{wolfer2021statistical} that if $\mtx{P}$ is geometrically ergodic, then for any $\varepsilon,\delta>0$ and any
\[n\geq c  \frac{t_{\text{mix}}}{\varepsilon^2} \max \left\{\|\mtx{D}_{\mu}\mtx{P}\|_{1 / 2}, \log \left(\frac{\|\nu/ \mu\|_{\mu}}{\delta}\right)\right\},\] $\|\mtx{R}_n-\mtx{D}_{\mu}\mtx{P}\|_{\text{TV}}<\varepsilon $ holds with probability at least $1-\delta$. Here $c$ is an absolute constant, $\|\mtx{R}_n-\mtx{D}_{\mu}\mtx{P}\|_{\text{TV}}$ is the total variation error between $\mtx{R}_n$ and $\mtx{D}_{\mu}\mtx{P}$ (interpreted as distributions over $\Omega^2$), and 
\[\|\mtx{D}_{\mu}\mtx{P}\|_{1/2}:=\left(\sum\limits_{u,v\in \Omega} (\mu(u)p(u,v))^{1/2}\right)^2.\] 
 It is not hard to see that $\|\mtx{D}_{\mu}\mtx{P}\|_{1/2}$ scales like $d^2$ in the worst case. Hence, the resulting sample complexity may depend on the dimension $d$.
 
To establish this result, they (in \cite{wolfer2021statistical}) first showed that in the stationary case, $\mathbb{E}[\|\mtx{R}_n-\mtx{D}_{\mu}\mtx{P}\|_{\text{TV}}]<\varepsilon/2$ for $n \geq c t_{\text{mix}}\|\mtx{D}_{\mu}\mtx{P}\|_{1/2} /\varepsilon^2$, and then invoked the McDiarmid's inequality for Markov chains established in \cite[Corollary 2.10]{paulin2015concentration} to control the deviation probability of $\|\mtx{R}_n-\mtx{D}_{\mu}\mtx{P}\|_{\text{TV}}$ around its expectation:
\begin{equation}\label{eqt:concentration of error norm}
    \Prob{ |\|\mtx{R}_n-\mtx{D}_{\mu}\mtx{P}\|_{\text{TV}}-\mathbb{E}[\|\mtx{R}_n-\mtx{D}_{\mu}\mtx{P}\|_{\text{TV}}]|>\varepsilon/2}\leq 2\exp\left(-\frac{n\varepsilon^2}{ct_{\text{mix}}}\right).
\end{equation}
Finally, the result was extended to the non-stationary case with an extra factor $\|\nu/ \mu\|_{\mu}$.

It is worth noting that \cite[Theorem 3.3]{wolfer2021statistical} only works when the transition matrix $\mtx{P}$ is geometrically ergodic. However, by the central limiting theorem for Markov chains, one does not need the geometric ergodicity of $\mtx{P}$ to guarantee the convergence of the Markov empirical mean to its space average. In comparison, our Theorem \ref{thm:main results of natural counting}(\ref{item:MLEmainresults tail bound}) establishes a  probabilistic error bound for MLE at the entry-wise level, which can be applied to any irreducible Markov chain (not necessarily geometrically ergodic). Note that $r_n(u,v)$ can be interpreted as the Markov empirical mean of the indicator function $\mathbb{1}_{(u,v)}(\cdot,\cdot)$ corresponding to $\mtx{P}_2$. Therefore, a recently established Bernstein-type inequality for Markov sums \cite{huang2026bernstein}, under milder assumptions than \cite[Corollary 2.10]{paulin2015concentration}, can be applied to $\mtx{P}_2$ to derive that 
\[\Prob{|r_n(u,v)-\mu(u)p(u,v)|>t}\lesssim \exp\left(-cn\eta_p(\mtx{P}_2) t^2\right).\]
We then establish Theorem \ref{thm:main results of natural counting}(\ref{item:MLEmainresults tail bound}) by further showing that $\eta_p(\mtx{P}_2)\geq \eta_p(\mtx{P})/(1+\eta_p(\mtx{P}))$ in Theorem~\ref{thm:eigenvalue of P_2}. 

One may also consider applying typical concentration inequalities for Markov sums~\cite{lezaud1998chernoff,leon2004optimal,miasojedow2014hoeffding,paulin2015concentration,jiang2018bernstein,fan2021hoeffding} to establish tail bounds similar to Theorem \ref{thm:main results of natural counting}(\ref{item:MLEmainresults tail bound}). This will require analyzing other spectral gaps of $\mtx{P}_2$ defined in Subsection~\ref{subsec:other spectral gaps}, which will also be presented in Theorem~\ref{thm:eigenvalue of P_2} as a result of independent interest.
Specifically, we show that $\eta_a(\mtx{P}_2)\equiv 0$ for any transition matrix $\mtx{P}$. Hence, most typical concentration results are inapplicable to $\mtx{P}_2$, except for \cite[Theorem 3.4]{paulin2015concentration}. On the other hand, we prove that $\eta_{ps}(\mtx{P}_2)\geq \eta_{ps}(\mtx{P})/2$, which allows us to apply \cite[Theorem 3.4]{paulin2015concentration} to deduce that
\begin{equation}\label{eqt:result_using_pseudo_gap}
    \Prob{|r_n(u,v)-\mu(u)p(u,v)|>t}\lesssim \exp\left(-cn\eta_{ps}(\mtx{P}) t^2\right)
\end{equation}
if $n\gtrsim 1/(\eta_{ps}(\mtx{P}))$. However, as shown in Subsection~\ref{subsec:other spectral gaps}, if $\mtx{P}$ is not geometrically ergodic, then it is possible that $\eta_{ps}(\mtx{P})=0$. Moreover, since $\eta_p(\mtx{P})\geq \eta_{ps}(\mtx{P})/2$, the bound in \eqref{eqt:result_using_pseudo_gap} does not provide any improvement over the spectral gap term of Theorem \ref{thm:main results of natural counting}(\ref{item:MLEmainresults tail bound}). In fact, since all other spectral gaps introduced in Subsection~\ref{subsec:other spectral gaps} can be bounded in terms of the IP gap, the only potential benefit of employing other concentration results instead of \cite{huang2026bernstein} lies in achieving a possibly better constant $c$ in the resulting tail bound.

Since Theorem \ref{thm:main results of natural counting}(\ref{item:MLEmainresults tail bound}) is an entry-wise result, it is natural to ask if we can establish a probabilistic error bound for MLE at the matrix level that depends on the IP gap of $\mtx{P}$. The main challenge in obtaining a result analogous to \cite[Theorem 3.3]{wolfer2021statistical} lies in studying the concentration of the error norm \eqref{eqt:concentration of error norm}, since Paulin's result \cite[Corollary 2.10]{paulin2015concentration} requires a finite mixing time, which is not guaranteed for all irreducible transition matrices. It would be interesting to develop a Markov concentration inequality for general functions of $n$ variables instead of sums based on the IP gap. If such a result is established, one could then derive a matrix-level bound analogous to~\cite[Theorem 3.3]{wolfer2021statistical}. We will leave it as a future work. 

On the other hand, when $\mtx{P}$ is reversible, the transition matrix $\tilde{\mtx{P}}_2$ defined on the symmetric path space $\tilde{\Omega}^2$ always admits a finite mixing time, which can be shown to be bounded by the mixing time of $(\mtx{P}+\mtx{\Id})/2$ plus 1 similar to \cite[Lemma 6.1]{wolfer2021statistical}. Therefore, following the proof strategy similar to \cite[Theorem 3.3]{wolfer2021statistical}, one can derive a matrix-level bound for SCE in terms of the mixing time of $(\mtx{P}+\mtx{\Id})/2$, which can further be  bounded by $c\log(1/\mu_*)/\eta(\mtx{P})$. Here $\mu_*:=\min_{u\in\Omega}\mu (u)>0$ by the Perron--Frobenius theorem.
We do not pursue this approach here, as the proof is essentially the same. Instead, we establish Theorem \ref{thm:main results of symmetric counting}(\ref{item:SCEmainresults tail bound}) by applying a recently established matrix concentration inequality \cite{neeman2024concentration} to  $\tilde{\mtx{P}}_2$, providing a matrix-level error bound for SCE via a more direct method. Note, however, that the result in \cite{neeman2024concentration} requires a nonzero absolute spectral gap of the transition matrix, which does not hold for $\mtx{P}_2$. Therefore, extending this approach to MLE would require a matrix concentration result based on more robust spectral gaps (e.g., the IP gap), which is another interesting direction for future research.

\section{Markov chains on path spaces}\label{sec:proofs}
This section focuses on the induced Markov chains on the path spaces. First, we construct Markov chains on $\Omega^2$ and $\tilde{\Omega}^2$ with $\mtx{P}_2$ and $\tilde{\mtx{P}}_2$ as their transition matrices respectively. Following this, the framework of MLE and SCE is proposed in detail. Lastly, we study the spectral properties of $\mtx{P}_2$ and $\tilde{\mtx{P}}_2$, which are crucial for establishing non-asymptotic error bounds for MLE and SCE in Section \ref{sec:error bounds}.
\subsection{Construction of \texorpdfstring{$\mtx{P}_2$}{$P_2$} and \texorpdfstring{$\tilde{\mtx{P}}_2$}{$\tilde{P}_2$}}\label{construction of Markov chain on path space}
Let us introduce the construction of the induced Markov chains on length-2 path spaces.
\subsubsection{Construction of \texorpdfstring{$\mtx{P}_2$}{$P_2$} }\label{subsubsec:construction of P2}
Recall that $\mtx{P}=[p(u, v)]_{u,v \in \Omega}$ is a transition matrix on the state space $\Omega$ with a unique invariant distribution $\mu$ . Then the length-2 path space $\Omega^2=\{(u_1,u_2):u_1,u_2\in \Omega\}$ is naturally equipped with a distribution $\mu_2$ that is given by
\[\mu_2\left(u_1, u_2\right)=\mu\left(u_1\right) p\left(u_1, u_2\right) , \quad u_i \in \Omega, i=1,2.\]
We will use bold letters, e.g. $\mtx{u}=(u_1,u_2)$ to denote elements in $\Omega^2$. This Markov transition law on $\Omega$ naturally induces a Markov transition law on $\Omega^2$, whose transition probabilities are as follows:
\begin{enumerate}
    \item \label{construction of P2:step 1}Given a state $\mtx{u}=(u_1,u_2)\in \Omega^2$, delete the first element $u_1$.
    \item \label{construction of P2:step 2}Choose a state $v\in \Omega$ according to the probability 
    \[\mathbb{P}\{\text {choose}~v\}=p\left(u_2, v\right),\]
    and then jump to the state $\mtx{v}=(u_2,v)$. 
\end{enumerate}   
We will use $\mtx{P}_2=[p_2(\mtx{u},\mtx{v})]_{\mtx{u},\mtx{v}\in \Omega^2}$ to denote the transition matrix of this Markov chain on $\Omega^2$. Later we will show that $\mu_2$ is the unique invariant distribution of $\mtx{P}_2$. 

In general, suppose $\{u_k\}_{k\geq 1}$ is a Markov sequence generated by $\mtx{P}$ with an initial distribution $u_1\sim \nu$. Then it is not hard to check that $\{\mtx{u}_k=(u_k,u_{k+1})\}_{k\geq 1}$ can be viewed as a Markov chain on $\Omega^2$ generated by $\mtx{P}_2$ with an initial distribution $\mtx{u}_1=(u_1,u_2)\sim\nu_2$, where
\[\nu_2(u_1,u_2)=\nu(u_1)p(u_1,u_2).\]

\subsubsection{Construction of \texorpdfstring{$\tilde{\mtx{P}}_2$}{$\tilde{P}_2$}}\label{subsubsec:construction of tildeP2}
Note that $\mtx{P}_2$ is non-reversible in general even if $\mtx{P}$ is reversible. Next we turn to the construction of a reversible Markov chain on the symmetric length-2 path space for a reversible $\mtx{P}$. For two paths $\mtx{u}=(u_1, u_2)$ and $\mtx{v}=(v_1,v_2) \in \Omega^2$, we say $\mtx{u}$ is equivalent to $\mtx{v}$, denoted by $\mtx{u} \sim \mtx{v}$, if 
\[\left(u_1,u_2\right)=\left(v_1,v_2\right)\quad \text{or}\quad \left(u_1,u_2\right)=\left(v_2,v_1\right).\]
We then define the symmetric path space as $\tilde{\Omega}^2:=\Omega^2/\sim$. For a path $\mtx{u} \in \Omega^2$ and an equivalence class $\tilde{\mtx{u}} \in \tilde{\Omega}^2$, we write $\mtx{u} \sim \tilde{\mtx{u}}$ or $\tilde{\mtx{u}} \sim \mtx{u}$ if $\mtx{u}$ falls in $\tilde{\mtx{u}}$. The distribution $\mu_2$ on $\Omega^2$ induces a distribution $\tilde{\mu}_2$ on $\tilde{\Omega}^2$:
\[\tilde{\mu}_2(\tilde{\mtx{u}}):=\sum\limits_{\mtx{u} \sim \tilde{\mtx{u}}} \mu_2(\mtx{u}), \quad \tilde{\mtx{u}} \in \tilde{\Omega}^2.\]
Correspondingly, we can construct a reversible Markov chain on the symmetric path space $\tilde{\Omega}^2$ as follows whose invariant distribution is exactly $\tilde{\mu}_2$:
\begin{enumerate}
    \item \label{construction of tildeP2:step 1} Given a state $\tilde{\mtx{u}}=\widetilde{(u_1,u_2)}\in \tilde{\Omega}^2$, randomly choose a path $\mtx{w}=(w_1,w_2) \in \Omega^2$ according to the probability
    \[\mathbb{P}\{w_1=u_1,w_2=u_2\}= \mathbb{P}\{w_1=u_2,w_2=u_1\}=\frac{1}{2}.\]
    \item \label{construction of tildeP2:step 2} Choose a state $v \in \Omega$ according to the probability
   \[\mathbb{P}\{\text {choose } v\}=p\left(w_2, v\right),\]
and then jump to the state $\tilde{\mtx{v}} \in \tilde{\Omega}^2$ such that $(w_2, v) \sim \tilde{\mtx{v}}$.
\end{enumerate}
We will use $\tilde{\mtx{P}}_2=[\tilde{p}_2(\tilde{\mtx{u}}, \tilde{\mtx{v}})]_{\tilde{\mtx{u}},\tilde{\mtx{v}}\in \tilde{\Omega}^2}$ to denote the transition matrix of this Markov chain on $\tilde{\Omega}^2$. Roughly speaking, $\tilde{\mtx{P}}_2$ can be seen as a modification of $\mtx{P}_2$ by randomly reversing the path with probability $1/2$, which makes the process reversible. Similarly, we will show that $\tilde{\mu}_2$ is the unique invariant distribution of $\tilde{\mtx{P}}_2$.

Previously, we can get a Markov sequence of $\mtx{P}_2$ directly from a Markov sequence of $\mtx{P}$. But this is not the case for $\tilde{\mtx{P}}_2$ due to the extra reversing operation. Nevertheless, from the construction of $\tilde{\mtx{P}}_2$ we can see that, as long as we have access to an online oracle of $\mtx{P}$ that takes in a state $u\in \Omega$ and outputs $v\in \Omega$ with probability $p(u,v)$ and an initial distribution $\nu$ for generating the first single state, we can get a Markov sequence $\{\tilde{\mtx{u}}_k\}_{k\geq 1}$ generated by $\tilde{\mtx{P}}_2$ with $\tilde{\mtx{u}}_1\sim\tilde{\nu}_2$, where
\[\tilde{\nu}_2\left(\widetilde{(u_1,u_2)}\right)=\begin{cases}\nu(u_1)p(u_1,u_2)+\nu(u_2)p(u_2,u_1),\quad  &\text{if} \,\ u_1\ne u_2,\\
    \nu(u_1)p(u_2,u_1), \quad   &\text{if} \,\ u_1= u_2.
    \end{cases}\]

\subsubsection{Invariant distributions}
The next lemma shows that $\mu_2$ and $\tilde{\mu}_2$ are indeed the invariant distributions of $\mtx{P}_2$ and $\tilde{\mtx{P}}_2$ respectively, and that $\tilde{\mtx{P}}_2$ is a reversible transition matrix. 
 
\begin{lemma}\label{reversibility of symmetric path space}
    Let $\mtx{P}$ be an irreducible transition matrix on $\Omega$ with an invariant distribution $\mu$. Then $\mu_2$ is an invariant distribution of $\mtx{P}_2$, i.e. 
    \[\mu_2(\mtx{u})=\sum\limits_{\mtx{v}\in \Omega^2}\mu_2(\mtx{v})p_2(\mtx{v},\mtx{u}), \quad \text{for any}\,\  \mtx{u}\in \Omega^2.\]
    Moreover, when $\mtx{P}$ is reversible, $\tilde{\mu}_2$ is an invariant distribution of $\tilde{\mtx{P}}_2$, and $\tilde{\mtx{P}}_2$ satisfies the detailed balance condition:
    \begin{equation}\label{eqt: reversibility of tildeP_2}
        \tilde{\mu}_2(\tilde{\mtx{u}})\tilde{p}_2(\tilde{\mtx{u}},\tilde{\mtx{v}})= \tilde{\mu}_2(\tilde{\mtx{v}})\tilde{p}_2(\tilde{\mtx{v}},\tilde{\mtx{u}}), \quad \text{for any}\,\ \tilde{\mtx{u}},\tilde{\mtx{v}}\in \tilde{\Omega}^2.
     \end{equation}
      \end{lemma}
\begin{proof}
    Since $\mu$ is the invariant distribution of $\mtx{P}$, we know \eqref{eqt:invariant distribution} holds for any $u,v\in \Omega$. Thus, for any $\mtx{u}=(u_1,u_2)\in \Omega^2$,  we can compute that
    \begin{align*}
        \sum\limits_{\mtx{v}\in \Omega^2}\mu_2(\mtx{v})p_2(\mtx{v},\mtx{u})&=\sum\limits_{v\in \Omega}\mu_2\left(v,u_1\right)p\left(u_1,u_2\right)\\ 
        &=\sum\limits_{v\in \Omega}\mu(v)p\left(v,u_1\right)p\left(u_1,u_2\right)\\
        &=\mu\left(u_1\right)p\left(u_1,u_2\right)\\
        &=\mu_2(\mtx{u}),
    \end{align*}
    which implies $\mu_2$ is the invariant distribution of $\mtx{P}_2$. 

    When $\mtx{P}$ is reversible, since \eqref{eqt: reversibility of tildeP_2} implies that $\tilde{\mu}_2$ is an invariant distribution of $\tilde{\mtx{P}}_2$, we only need to show the reversibility of $\tilde{\mtx{P}}_2$ under $\tilde{\mu}_2$. Notice that \eqref{eqt: reversibility of tildeP_2} always holds for $\tilde{\mtx{u}}=\tilde{\mtx{v}}$. On the other hand, for any $\tilde{\mtx{u}}\ne\tilde{\mtx{v}}$ which satisfy $\tilde{p}_2(\tilde{\mtx{u}},\tilde{\mtx{v}})\ne 0$, by the definition of $\tilde{\mtx{P}}_2$, we can assume $\tilde{\mtx{u}}\sim (u,w)$, $\tilde{\mtx{v}}\sim (v,w)$, $u,v,w\in \Omega$.  We then use the reversibility of $\mtx{P}$ \eqref{eqt:detailed balance} to deduce that
    \[\tilde{\mu}_2(\tilde{\mtx{u}})\tilde{p}_2(\tilde{\mtx{u}},\tilde{\mtx{v}})=\mu(u)p(u,w)p(w,v)=\mu(v)p(v,w)p(w,u)= \tilde{\mu}_2(\tilde{\mtx{v}})\tilde{p}_2(\tilde{\mtx{v}},\tilde{\mtx{u}}),\]
    which implies the reversibility of $\tilde{\mtx{P}}_2$.
\end{proof}

\subsection{Unbiased estimation}
Now that we have formally introduced the induced Markov chains on path spaces, we can elaborate the connection between MLE and $\mtx{P}_2$. By \eqref{eqt:formula of MLE} and the definition of $\mtx{P}_2$, the MLE estimator $\mtx{R}_n=[r_n(u,v)]_{u,v\in\Omega}$ can be identified as a matrix empirical mean generated by $\mtx{P}_2$:
\[\mtx{R}_n=\frac{1}{n}\sum\limits_{k=1}^n\mtx{F}(\mtx{u}_k),\]
where $\mtx{F}:\Omega^2\rightarrow \R^{d\times d} $ is a matrix-valued function defined by
\[\mtx{F}(\mtx{w}):=\mtx{E}_{w_1w_2},\quad \text{for any} \,\ \mtx{w}=(w_1,w_2)\in\Omega^2.\]
That is, the $(u,v)$ entry of $\mtx{F}$ is just the indicator function $\mathbb{1}_{(u,v)}$ of the pair $(u,v)\in \Omega^2$, i.e. $\mtx{F}=[\mathbb{1}_{(u,v)}]_{u,v\in \Omega}$.

Our SCE is associated with $\tilde{\mtx{P}}_2$ in a similar way. More precisely, given a Markov sequence $\{\tilde{\mtx{u}}_k\}_{k\geq 1}$ generated by $\tilde{\mtx{P}}_2$, we can approximate $\mtx{D}_{\mu}\mtx{P}$ using the corresponding matrix empirical mean  
\[\mtx{H}_n=\frac{1}{n}\sum\limits_{k=1}^n\tilde{\mtx{F}}(\tilde{\mtx{u}}_k),\]
where $\tilde{\mtx{F}}:\tilde{\Omega}^2 \rightarrow \mathbb{H}_d $ is a matrix-valued function defined by
\[\tilde{\mtx{F}}(\tilde{\mtx{w}}):=\frac{1}{2}\left(\mtx{E}_{w_1w_2}+\mtx{E}_{w_2w_1}\right), \quad \text{for any}\,\ \tilde{\mtx{w}}\sim (w_1,w_2).\]
Let $\tilde{f}_{uv}$ denote the $(u,v)$ entry of $\tilde{\mtx{F}}$. Then
\[\tilde{f}_{uv}=\tilde{f}_{vu}=
    \begin{cases}
      \frac{1}{2}\mathbb{1}_{\widetilde{(u,v)}},  &\text{if}\,\ u\ne v,\\
      \mathbb{1}_{\widetilde{(u,v)}},   &\text{if}\,\ u=v,
    \end{cases}\]
where $\mathbb{1}_{\widetilde{(u,v)}}$ is the indicator function of the symmetric pair $\widetilde{(u,v)}$.

The next lemma verifies that MLE and SCE are unbiased when the corresponding Markov chains on path spaces start from their invariant distributions.
\begin{lemma}\label{lem:unbiased of F and tildeF}
    It holds that
    \begin{equation}\label{eqt:expectation of F}
        \mathbb{E}_{\mu_2}\left[\mtx{F}\right]=\mtx{D}_{\mu}\mtx{P}.
    \end{equation}
    Furthermore, when $\mtx{P}$ is reversible,
    \begin{equation}\label{eqt:expectation of tildeF}
        \mathbb{E}_{\tilde{\mu}_2}\left[\tilde{\mtx{F}}\right]=\mtx{D}_{\mu}\mtx{P}.
    \end{equation}
\end{lemma}
\begin{proof}
    We use the definition of $\mtx{F}=[\mathbb{1}_{(u,v)}]_{u,v\in \Omega}$ to compute that
    \[\mathbb{E}_{\mu_2}\left[\mathbb{1}_{(u,v)}\right]=\mu_2(u,v)=\mu(u)p(u,v),\]
    which is the entry-wise form of \eqref{eqt:expectation of F}.

    When $\mtx{P}$ satisfies the detailed balance condition \eqref{eqt:detailed balance}, for any $u,v\in \Omega$ such that $u\ne v$, we have
    \[\E_{\tilde{\mu}_2}\left[\tilde{f}_{uv}\right]=\frac{1}{2}\tilde{\mu}_2(\widetilde{(u,v)})=\frac{1}{2}(\mu(u)p(u,v)+\mu(v)p(v,u))=\mu(u)p(u,v).\]
    Moreover, for any $u\in \Omega$,
    \[\E_{\tilde{\mu}_2}\left[\tilde{f}_{uu}\right]=\tilde{\mu}_2(\widetilde{(u,u)})=\mu(u)p(u,u).\]
    The above calculations imply \eqref{eqt:expectation of tildeF}.
\end{proof}

In the next section, we will study the convergence of $\mtx{R}_n$ and $\mtx{H}_n$ and establish non-asymptotic error bounds for MLE and SCE, before which we respectively propose the detailed frameworks of the two methods in Algorithm~\ref{alg:ML} and Algorithm~\ref{alg:SC} as follows.
\begin{algorithm}[] 
    \caption{Maximal Likelihood Estimation}\label{alg:ML} 
    \LinesNumbered 
    \KwIn{A sequence $\{u_k\}$ generated by $\mtx{P}$ with length $N+1$;}
    $k=1$\;
    $\mtx{R}=\sum\limits_{k=1}^N(\mtx{E}_{u_k u_{k+1}})/N$\;
    $\hat{\mu}=\mtx{R}\mtx{1}$\;
    $\hat{\mtx{P}}=\mtx{D}_{\hat{\mu}}^{-1}\mtx{R}$
\end{algorithm}

\begin{algorithm}[] 
    \caption{Symmetric Counting Estimation}\label{alg:SC} 
    \LinesNumbered 
    \KwIn{An oracle $\mathcal{P}$ of a reversible Markov chain, initial distribution $\nu$, total iteration number $N$;}
    $k=1$\;
    generate $u$ according to $\nu$\;
    $v=\mathcal{P}(u)$\;
    $\mtx{H}=(\mtx{E}_{uv}+\mtx{E}_{vu})/(2N)$\;
    \For{$k=2:N$}{
        generate $r \sim \mathcal{U}[0,1]$\;
       \If{$r\leq \frac{1}{2}$}
       {$u\leftarrow \mathcal{P}(v)$;}
       \Else{$v\leftarrow \mathcal{P}(u)$;}
       $\mtx{H}\leftarrow \mtx{H}+(\mtx{E}_{uv}+\mtx{E}_{vu})/(2N)$\;
    }
    $\hat{\mu}=\mtx{H}\mtx{1}$\;
    $\hat{\mtx{P}}=\mtx{D}_{\hat{\mu}}^{-1}\mtx{H}$
\end{algorithm}

\subsection{Spectral properties of \texorpdfstring{$\mtx{P}_2$}{$P_2$} and \texorpdfstring{$\tilde{\mtx{P}}_2$}{$\tilde{P}_2$}}\label{properties of P_2}
In order to derive non-asymptotic error bounds for MLE and SCE, we need to study the spectral gaps of $\mtx{P}_2$ and $\tilde{\mtx{P}}_2$, respectively. To start, the next theorem focuses on the spectral properties of $\mtx{P}_2$. We will show that $\mtx{P}$ and $\mtx{P}_2$ share the same nonzero eigenvalues with algebraic multiplicity accounted by decomposing $\mtx{P}_2$ as the product of two low-rank matrices according to the construction of $\mtx{P}_2$.
We further show that the pseudo spectral gap, the symmetric spectral gap and the IP gap of $\mtx{P}_2$ can all be bounded from below in terms of the corresponding spectral gaps of $\mtx{P}$, however, the absolute spectral gap of $\mtx{P}_2$ is always zero.

\begin{theorem}\label{thm:eigenvalue of P_2}
    Let $\mtx{P}$ be an irreducible transition matrix on $\Omega$ with an invariant distribution $\mu$. The followings hold:
    \begin{enumerate}
        \item $\mtx{P}$ and $\mtx{P}_2$ share the same nonzero eigenvalues with algebraic multiplicity accounted.
        \item For every $k\in\mathbb{N}_+$, $\mtx{P}^{k-1}$ and $\mtx{P}_2^k$ share the same nonzero singular values (with respect to $L_{2,\mu}$ and $L_{2,\mu_2}$, respectively) with multiplicity accounted. As a result, $\eta_{a}(\mtx{P}_2)\equiv0$, $\eta_{ps}(\mtx{P}_2)\geq \eta_{ps}(\mtx{P})/2$.
        \item $ \eta_s(\mtx{P}_2)\geq\eta_s(\mtx{P})/2>0$, and hence, $\mu_2$ is the unique invariant distribution of $\mtx{P}_2$. Moreover, when $\mtx{P}$ is reversible, any nonzero eigenvalue of $(\mtx{P}_2+\mtx{P}^*_2)/2$ must be a nonzero eigenvalue of $(\mtx{P}+\mtx{\Id})/2$ or $(\mtx{P}-\mtx{\Id})/2$ with multiplicity accounted, and vice versa. As a result, $\eta_s(\mtx{P}_2)=\eta(\mtx{P})/2$.
        \item  $\eta_p(\mtx{P}_2)\geq\eta_p(\mtx{P})/(1+\eta_p(\mtx{P}))>0$.
    \end{enumerate}
\end{theorem}
\begin{proof}
    From the construction of $\mtx{P}_2$ we find that $\mtx{P}_2$ can be decomposed as the product of two transition matrices that represent step \ref{construction of P2:step 1} and step \ref{construction of P2:step 2} introduced in Subsection~\ref{subsubsec:construction of P2}, respectively.
    Let $\mtx{S}$ denote the transition matrix of step \ref{construction of P2:step 1} in which a given state $\mtx{u}=(u_1,u_2)\in \Omega^2$ will jump to $u_2$ directly, i.e. for any $\mtx{u}=(u_1,u_2)\in \Omega^2$, $v\in \Omega$,
    \[\mtx{S}(\mtx{u},v)=\delta_{u_2v},\]
    where $\delta_{uv}$ represents the Dirac symbol:
    \[\delta_{uv}=
    \begin{cases} 
        1, \quad \text{if}\,\ u=v,\\
        0, \quad \text{if}\,\ u\ne v.
    \end{cases}\]
    Let $\mtx{T}$ denote the transition matrix of step \ref{construction of P2:step 2} in which a given state $u\in \Omega$ will jump to the state $(u,v)\in \Omega^2$ according to the probability 
    \[\mathbb{P}\left\{\text{choose } v \right\}=p(u,v),\]
    i.e. for any $u\in \Omega$, $\mtx{v}=(v_1,v_2)\in \Omega^2$,  
    \[\mtx{T}(u,\mtx{v})=p(u,v_2)\delta_{uv_1}=p(v_1,v_2)\delta_{uv_1}.\]
    By the definition of $\mtx{P}_2$ we know $\mtx{P}_2=\mtx{S}\mtx{T}$. On the other hand, we find that, for any $u,v\in\Omega$,
    \[\mtx{T}\mtx{S}(u,v)=\sum\limits_{\mtx{u}\in\Omega^2}\mtx{T}(u,\mtx{u})\mtx{S}(\mtx{u},v)=p(u,v),\]
    that is, $\mtx{T}\mtx{S}=\mtx{P}$. We therefore deduce that $\mtx{P}$ and $\mtx{P}_2$ share the same nonzero eigenvalues from the basic fact that $\mtx{S}\mtx{T}$ and $\mtx{T}\mtx{S}$ share the same nonzero eigenvalues.

    In order to obtain the absolute spectral gap and pseudo spectral gap of $\mtx{P}_2$, we need to study the singular values of the matrix ($\mtx{D}_{\mu_2}^{1/2}\mtx{P}_2\mtx{D}_{\mu_2}^{-1/2})^k $. Let $\mtx{S}_1=\mtx{D}_{\mu_2}^{1/2}{\mtx{S}}\mtx{D}_{\mu}^{-1/2}$ and $\mtx{T}_1=\mtx{D}_{\mu}^{1/2}{\mtx{T}}\mtx{D}_{\mu_2}^{-1/2}$. By the definitions of $\mtx{S}_1$ and $\mtx{T}_1$ we have
    \[\mtx{S}_1\mtx{T}_1 =\mtx{D}_{\mu_2}^{1/2}\mtx{P}_2\mtx{D}_{\mu_2}^{-1/2} ,\quad \mtx{T}_1\mtx{S}_1=\mtx{D}_{\mu}^{1/2}\mtx{P}\mtx{D}_{\mu}^{-1/2}.\]
    For any $u,v\in \Omega$, we use the definition of the invariant distribution $\mu$ \eqref{eqt:invariant distribution} to compute that
    \begin{align*}
        \left({\mtx{S}_1}^{\mathrm{T}}{\mtx{S}_1}\right)(u,v)&=\mu(u)^{-1/2}\mu(v)^{-1/2}\sum\limits_{\mtx{u}=(u_1,u_2)\in {\Omega}^2}{\mtx{S}}(\mtx{u},u){\mu}_2(\mtx{u}){\mtx{S}}(\mtx{u},v)\\
        &= \mu(u)^{-1/2}\mu(v)^{-1/2} \sum\limits_{\mtx{u}=(u_1,u_2)\in {\Omega}^2} \delta_{u_2u}\delta_{u_2v} {\mu}_2(\mtx{u})\\
        & =\delta_{uv}\mu(u)^{-1}\sum\limits_{u_1\in {\Omega}}\mu(u_1)p(u_1,u)\\
        &=\delta_{uv},
    \end{align*}
    and 
    \begin{align*}
        \left({\mtx{T}_1}{\mtx{T}_1}^{\mathrm{T}}\right)(u,v)&=\mu(u)^{1/2}\mu(v)^{1/2}\sum\limits_{\mtx{u}=(u_1,u_2)\in {\Omega}^2}{\mtx{T}}(u,\mtx{u}){\mu}_2(\mtx{u})^{-1}{\mtx{T}}(v,\mtx{u})\\
        &=\mu(u)^{1/2}\mu(v)^{1/2}\sum\limits_{\mtx{u}=(u_1,u_2)\in {\Omega}^2}\delta_{uu_1}\delta_{vu_1}p(u,u_2)p(v,u_2){\mu}_2(\mtx{u})^{-1}\\
        &=\delta_{uv}\mu(u)\sum\limits_{u_2\in \Omega}\frac{p(u,u_2)^2}{\mu(u)p(u,u_2)}\\
        &=\delta_{uv}.
    \end{align*}
    We thus deduce $\mtx{S}_1^{\mathrm{T}}\mtx{S}_1=\mtx{T}_1\mtx{T}_1^{\mathrm{T}}=\mtx{\Id}$.
 Note that, for all $k\in \N_+$, \[\left(\mtx{D}_{\mu_2}^{1/2}\mtx{P}_2^k\mtx{D}_{\mu_2}^{-1/2}\right)^{\mathrm{T}}\left(\mtx{D}_{\mu_2}^{1/2}\mtx{P}_2^k\mtx{D}_{\mu_2}^{-1/2}\right)= \left(\mtx{T}_1^{\mathrm{T}}\mtx{S}_1^{\mathrm{T}}\right)^{k}\left(\mtx{S}_1\mtx{T}_1\right)^{k} \]
 and \begin{align*}
    \left(\mtx{D}_{\mu}^{1/2}\mtx{P}^{k-1}\mtx{D}_{\mu}^{-1/2}\right)^{\mathrm{T}}\left(\mtx{D}_{\mu}^{1/2}\mtx{P}^{k-1}\mtx{D}_{\mu}^{-1/2}\right)& =\left(\mtx{S}_1^{\mathrm{T}}\mtx{T}_1^{\mathrm{T}}\right)^{k-1}\left(\mtx{T}_1\mtx{S}_1\right)^{k-1}\\
    &= \mtx{T}_1\mtx{T}_1^{\mathrm{T}}\left(\mtx{S}_1^{\mathrm{T}}\mtx{T}_1^{\mathrm{T}}\right)^{k-1}\mtx{S}_1^{\mathrm{T}}\mtx{S}_1 \left(\mtx{T}_1\mtx{S}_1\right)^{k-1}\\
    &=\mtx{T}_1\left(\mtx{T}_1^{\mathrm{T}}\mtx{S}_1^{\mathrm{T}}\right)^{k}\left(\mtx{S}_1\mtx{T}_1\right)^{k-1}\mtx{S}_1.
 \end{align*}
 Since $\left(\mtx{T}_1^{\mathrm{T}}\mtx{S}_1^{\mathrm{T}}\right)^{k}\left(\mtx{S}_1\mtx{T}_1\right)^{k}$ and $\mtx{T}_1\left(\mtx{T}_1^{\mathrm{T}}\mtx{S}_1^{\mathrm{T}}\right)^{k}\left(\mtx{S}_1\mtx{T}_1\right)^{k-1}\mtx{S}_1$ share the same nonzero eigenvalues, we deduce that $\mtx{P}^{k-1}$ and $\mtx{P}_2^k$ share the same nonzero singular values (with respect to $L_{2,\mu}$ and $L_{2,\mu_2}$, respectively). We therefore have \[\lambda_a\left(\mtx{P}_2^k\right)=\lambda_a\left(\mtx{P}^{k-1}\right).\]
    In particular, by taking $k=1$, we deduce $\lambda_a(\mtx{P}_2)=1$ and $\eta_a(\mtx{P}_2)=1-\lambda_a(\mtx{P}_2)=0$. By the definition of the pseudo spectral gap, we know that \[\eta_{ps}(\mtx{P})=\sup_{k\in \N_+} \frac{1}{k}\left(1-\lambda_a^2\left(\mtx{P}^k\right)\right).\]
    Suppose the above supremum is attained at $k_0\in [1,+\infty]$, we have
    \[\eta_{ps}(\mtx{P}_2)\geq \frac{1}{k_0+1}\left(1-\lambda_a^2\left(\mtx{P}_2^{k_0+1}\right)\right)=\frac{1}{k_0+1}\left(1-\lambda_a^2\left(\mtx{P}^{k_0}\right)\right)=\frac{k_0}{k_0+1}\eta_{ps}(\mtx{P})\geq \frac{\eta_{ps}(\mtx{P})}{2},\]
    which is the desired result.

    Next, we turn to study the symmetric spectral gap of $\mtx{P}_2$. Recall that $\lambda_s(\mtx{P}_2)$ is the second largest eigenvalue of the matrix $(\mtx{S}_1\mtx{T}_1+\mtx{T}_1^{\mathrm{T}}\mtx{S}_1^{\mathrm{T}})/2=:\mtx{K}$. Let us also define $\mtx{L}:= (\mtx{T}_1^{\mathrm{T}}\mtx{T}_1+\mtx{S}_1\mtx{S}_1^{\mathrm{T}})/2$. To study the spectrum of $\mtx{K}$, we first compute that
    \[\frac{1}{2}\left(\mtx{K}+\mtx{L}\right) =\frac{1}{4}\left(\mtx{S}_1\mtx{T}_1+\mtx{T}_1^{\mathrm{T}}\mtx{S}_1^{\mathrm{T}}\right)+\frac{1}{4}\left(\mtx{T}_1^{\mathrm{T}}\mtx{T}_1+\mtx{S}_1\mtx{S}_1^{\mathrm{T}}\right)=\frac{1}{4}\left(\mtx{T}_1^{\mathrm{T}}+\mtx{S}_1\right)\left(\mtx{T}_1+\mtx{S}_1^{\mathrm{T}}\right) \succeq 0,\]
    and
    \[\frac{1}{2}\left(\mtx{K}-\mtx{L}\right)=\frac{1}{4}\left(\mtx{S}_1\mtx{T}_1+\mtx{T}_1^{\mathrm{T}}\mtx{S}_1^{\mathrm{T}}\right)-\frac{1}{4}\left(\mtx{T}_1^{\mathrm{T}}\mtx{T}_1+\mtx{S}_1\mtx{S}_1^{\mathrm{T}}\right)=-\frac{1}{4}\left(\mtx{T}_1^{\mathrm{T}}-\mtx{S}_1\right)\left(\mtx{T}_1-\mtx{S}_1^{\mathrm{T}}\right)\preceq 0.\]
    As a result, we have 
    \[\frac{1}{2}\left(\mtx{K}-\mtx{L}\right)\preceq \mtx{K}=  \frac{1}{2}\left(\mtx{K}-\mtx{L}\right)+ \frac{1}{2}\left(\mtx{K}+\mtx{L}\right)\preceq \frac{1}{2}\left(\mtx{K}+\mtx{L}\right).\]
  This implies the eigenvalues of $\mtx{K}$, in non-decreasing order, are less than or equal to the corresponding eigenvalues of $(\mtx{K}+\mtx{L})/2$ and larger than or equal to those of $(\mtx{K}-\mtx{L})/2$.
    On the other hand, by exchanging the order of multiplication, we have
    \begin{equation}\label{eqt:P+I}
        \begin{aligned}
            \frac{1}{4}\left(\mtx{T}_1+\mtx{S}_1^{\mathrm{T}}\right)\left(\mtx{T}_1^{\mathrm{T}}+\mtx{S}_1\right)
            &=\frac{1}{4}\left(\mtx{T}_1\mtx{S}_1+\mtx{S}_1^{\mathrm{T}}\mtx{T}_1^{\mathrm{T}}\right)+\frac{1}{4}\left(\mtx{S}_1^{\mathrm{T}}\mtx{S}_1+\mtx{T}_1\mtx{T}_1^{\mathrm{T}}\right)\\
            &=\frac{1}{2}\left(\mtx{D}_{\mu}^{1/2}\frac{\mtx{P}+\mtx{\Id}}{2}\mtx{D}_{\mu}^{-1/2}+\mtx{D}_{\mu}^{-1/2}\frac{\mtx{P}^{\mathrm{T}}+\mtx{\Id}}{2}\mtx{D}_{\mu}^{1/2}\right),
        \end{aligned}
    \end{equation}
    and 
    \begin{equation}\label{eqt:P-I}
        -\frac{1}{4}\left(\mtx{T}_1-\mtx{S}_1^{\mathrm{T}}\right)\left(\mtx{T}_1^{\mathrm{T}}-\mtx{S}_1\right)=\frac{1}{2}\left(\mtx{D}_{\mu}^{1/2}\frac{\mtx{P}-\mtx{\Id}}{2}\mtx{D}_{\mu}^{-1/2}+\mtx{D}_{\mu}^{-1/2}\frac{\mtx{P}^{\mathrm{T}}-\mtx{\Id}}{2}\mtx{D}_{\mu}^{1/2}\right).
    \end{equation}
    For any real symmetric matrix $\mtx{M}$, we use $\lambda_2(\mtx{M})$ to denote the second largest eigenvalue of $\mtx{M}$. Since $(\mtx{T}_1+\mtx{S}_1^{\mathrm{T}})(\mtx{T}_1^{\mathrm{T}}+\mtx{S}_1)/4$ and $ (\mtx{T}_1^{\mathrm{T}}+\mtx{S}_1)(\mtx{T}_1+\mtx{S}_1^{\mathrm{T}})/4$ share the same nonzero eigenvalues, we deduce that
    \begin{align*}
        \lambda_s(\mtx{P}_2)&=\lambda_2(\mtx{K})\leq \lambda_2\left(\frac{\mtx{K}+\mtx{L}}{2}\right)=\lambda_2\left(\mtx{D}_{\mu}^{1/2}\frac{\mtx{P}+\mtx{\Id}}{4}\mtx{D}_{\mu}^{-1/2}+\mtx{D}_{\mu}^{-1/2}\frac{\mtx{P}^{\mathrm{T}}+\mtx{\Id}}{4}\mtx{D}_{\mu}^{1/2}\right)\\
        & = \lambda_s\left(\frac{\mtx{P}+\mtx{\Id}}{2}\right),
    \end{align*}
    and therefore
    \[\eta_s(\mtx{P}_2)=1-\lambda_s(\mtx{P}_2)\geq 1-\lambda_s\left(\frac{\mtx{P}+\mtx{\Id}}{2}\right)=\frac{\eta_s(\mtx{P})}{2}>0,\]
    where the last inequality is due to the irreducibility of $\mtx{P}$. As a direct corollary, we know $1$ is a simple eigenvalue of $\mtx{P}_2$ since
    \[\langle h, \left(\mtx{\Id}-\mtx{P}_2\right) h\rangle_{\mu_2}\geq \eta_s(\mtx{P}_2) \left\| h  \right\|_{\mu_2}^2>0, \qquad \text{for any } h\in L_{2,\mu_2}^0\backslash\{\mtx{0}\},\]
    which implies that $\mu_2$ is the unique invariant distribution of $\mtx{P}_2$.

    When $\mtx{P}$ is reversible, we have
    \[\mtx{T}_1\mtx{S}_1=\mtx{D}_{\mu}^{1/2}\mtx{P}\mtx{D}_{\mu}^{-1/2}=\mtx{D}_{\mu}^{-1/2}\mtx{P}^{\mathrm{T}}\mtx{D}_{\mu}^{1/2}=\mtx{S}_1^\mathrm{T}\mtx{T}_1^\mathrm{T},\]
    thus,
    \begin{align*}
        \mtx{K}^2
        &=\frac{1}{4}\left(\mtx{S}_1\mtx{T}_1\mtx{S}_1\mtx{T}_1+\mtx{S}_1\mtx{T}_1\mtx{T}_1^{\mathrm{T}}\mtx{S}_1^{\mathrm{T}}+\mtx{T}_1^{\mathrm{T}}\mtx{S}_1^{\mathrm{T}}\mtx{S}_1\mtx{T}_1+\mtx{T}_1^{\mathrm{T}}\mtx{S}_1^{\mathrm{T}}\mtx{T}_1^{\mathrm{T}}\mtx{S}_1^{\mathrm{T}}\right)\\
        &=\frac{1}{4}\left(\mtx{S}_1\mtx{T}_1\mtx{S}_1\mtx{T}_1+\mtx{T}_1^{\mathrm{T}}\mtx{S}_1^{\mathrm{T}}\mtx{T}_1^{\mathrm{T}}\mtx{S}_1^{\mathrm{T}}+\mtx{T}_1^{\mathrm{T}}\mtx{T}_1+\mtx{S}_1\mtx{S}_1^{\mathrm{T}}\right)\\
        &=\frac{1}{4}\left(\mtx{S}_1\mtx{S}_1^{\mathrm{T}}\mtx{T}_1^{\mathrm{T}}\mtx{T}_1+ \mtx{T}_1^{\mathrm{T}}\mtx{T}_1\mtx{S}_1\mtx{S}_1^{\mathrm{T}}  +   \mtx{T}_1^{\mathrm{T}}\mtx{T}_1\mtx{T}_1^{\mathrm{T}}\mtx{T}_1+\mtx{S}_1\mtx{S}_1^{\mathrm{T}}\mtx{S}_1\mtx{S}_1^{\mathrm{T}}\right)\\
        &=\mtx{L}^2.
    \end{align*}
    By the fact that $\mtx{K}\in \mathbb{H}_{d^2}$, $\mtx{L}\in \mathbb{H}_{d^2}^+$, we know $\mtx{L}$ must be the unique square root of $\mtx{K}^2$, which implies that if $(\lambda,\alpha)$ is an eigen-pair of $\mtx{K}$, then $(|\lambda|,\alpha)$ is an eigen-pair of $\mtx{L}$. Therefore, any nonzero eigenvalue of $\mtx{K}$ must be a nonzero eigenvalue of $(\mtx{K}+\mtx{L})/2$ or $(\mtx{K}-\mtx{L})/2$, and vice versa. Moreover, since $(\mtx{K}-\mtx{L})/2\preceq 0$ and $(\mtx{K}+\mtx{L})/2\succeq 0$, we know that if $\lambda_2((\mtx{K}+\mtx{L})/2)>0$, then it must also be the second largest eigenvalue of $\mtx{K}$. On the other hand, if $\lambda_2((\mtx{K}+\mtx{L})/2)=0$, then 1 is the only nonzero eigenvalue of $(\mtx{K}+\mtx{L})/2$, and thus $\operatorname{rank}((\mtx{K}+\mtx{L})/2)=1$. We also know that 
    \[\operatorname{rank}\left(\frac{\mtx{K}-\mtx{L}}{2}\right)= \operatorname{rank}\left(-\frac{1}{4}\left(\mtx{T}_1^{\mathrm{T}}-\mtx{S}_1\right)\left(\mtx{T}_1-\mtx{S}_1^{\mathrm{T}}\right)\right)\leq \operatorname{rank}\left(\mtx{T}_1^{\mathrm{T}}-\mtx{S}_1\right)\leq d,\]
    Therefore,
    \[\operatorname{rank}(\mtx{K})=\operatorname{rank}\left(\frac{\mtx{K}+\mtx{L}}{2}+\frac{\mtx{K}-\mtx{L}}{2}\right)\leq \operatorname{rank}\left(\frac{\mtx{K}+\mtx{L}}{2}\right)+\operatorname{rank}\left(\frac{\mtx{K}-\mtx{L}}{2}\right)\leq 1+d<d^2,\]
    which implies that 0 must be an eigenvalue of $\mtx{K}$. To summarize, the following equality always holds:
    \[\lambda_2(\mtx{K})=\lambda_2\left(\frac{\mtx{K}+\mtx{L}}{2}\right).\]
    To complete the proof, notice that in the reversible case, \eqref{eqt:P+I} and \eqref{eqt:P-I} respectively reduce to
    \begin{equation}\label{eqt:important relation}
        \frac{1}{4}\left(\mtx{T}_1+\mtx{S}_1^{\mathrm{T}}\right)\left(\mtx{T}_1^{\mathrm{T}}+\mtx{S}_1\right)=\mtx{D}_{\mu}^{1/2}\frac{\mtx{P}+\mtx{\Id}}{2}\mtx{D}_{\mu}^{-1/2},
    \end{equation}
    and 
    \[-\frac{1}{4}\left(\mtx{T}_1-\mtx{S}_1^{\mathrm{T}}\right)\left(\mtx{T}_1^{\mathrm{T}}-\mtx{S}_1\right)=\mtx{D}_{\mu}^{1/2}\frac{\mtx{P}-\mtx{\Id}}{2}\mtx{D}_{\mu}^{-1/2}.\]
    As a result, any nonzero eigenvalue of $\mtx{K}$ must be a nonzero eigenvalue of $(\mtx{P}+\mtx{\Id})/2$ or $(\mtx{P}-\mtx{\Id})/2$ with multiplicity accounted, and vice versa. Besides, by \eqref{eqt:important relation}, we have
    \[\lambda_s(\mtx{P}_2)=\lambda_2(\mtx{K})=\lambda_2\left(\frac{\mtx{K}+\mtx{L}}{2}\right)= \lambda\left(\frac{\mtx{P}+\mtx{\Id}}{2}\right),\]
    which implies $\eta_s(\mtx{P}_2)=\eta(\mtx{P})/2$.

    Finally, we study the IP gap of $\mtx{P}_2$. For any function $h\in L_{2,\mu_2}^0$, let $f:=(\mtx{\Id}-\mtx{P}_2)h$. Our goal is to bound $\|f\|_{\mu_2}/\|h\|_{\mu_2}$. Recall that $\{u_k\}_{k\geq 1}$ is the Markov chain generated by $\mtx{P}$ and $\{\mtx{u}_k\}_{k\geq 1}=\{(u_k,u_{k+1})\}_{k\geq 1}$ is the Markov chain generated by $\mtx{P}_2$. By the transition law of $\mtx{P}_2$, for any $\mtx{u}=(u,v)\in \Omega^2 $,
    \[\mtx{P}_2 h (\mtx{u})=\E [h(\mtx{u}_2)|\mtx{u}_1=\mtx{u}]=\E [h(v,u_3)|u_2=v]=:h_1(v).\]
    We deduce that \[f(u,v)=h(u,v)-h_1(v), \qquad \text{for any }u,v\in \Omega.\]
    We then fix the first variable $u$ and take expectation with respect to the second variable $v$ according to the transition law $p(u,v)$ to get
    \[\E[f(u,u_2)|u_1=u]= \E[h(u,u_2)|u_1=u]- \E[h_1(u_2)|u_1=u]=(\mtx{\Id}-\mtx{P})h_1(u).\]
    Since $\mtx{P}$ admits an IP gap $\eta_p(\mtx{P})>0$, we have
    \begin{align*}
        \|h_1\|_{\mu}&\leq \frac{1}{\eta_p(\mtx{P})}\|\E[f(\cdot,u_2)|u_1=\cdot]\|_{\mu}\\
        &=\frac{1}{\eta_p(\mtx{P})}\left(\E\left[\E[f(u,u_2)|u_1=u]^2|u\sim \mu\right]\right)^{1/2}\\
        &\leq \frac{1}{\eta_p(\mtx{P})}\left(\E\left[\E[f^2(u,u_2)|u_1=u]|u\sim \mu\right]\right)^{1/2}\\
        &=\frac{\|f\|_{\mu_2}}{\eta_p(\mtx{P})}.
    \end{align*}
    As a result,
    \[\|f\|_{\mu_2}=\|h-h_1\|_{\mu_2}\geq \|h\|_{\mu_2}-\|h_1\|_{\mu_2}=\|h\|_{\mu_2}-\|h_1\|_{\mu}\geq \|h\|_{\mu_2}-\frac{\|f\|_{\mu_2}}{\eta_p(\mtx{P})},\]
    which implies
    \[\frac{\|(\mtx{\Id}-\mtx{P}_2)h\|_{\mu_2}}{ \|h\|_{\mu_2}}=\frac{\|f\|_{\mu_2}}{ \|h\|_{\mu_2}}\geq \frac{\eta_p(\mtx{P})}{1+\eta_p(\mtx{P})}.\]
    This completes the proof.
\end{proof}

Next, we study the spectrum of $\tilde{\mtx{P}}_2$ when $\mtx{P}$ is reversible. Compared with Theorem~\ref{thm:eigenvalue of P_2}, we can obtain the spectral gaps of $\tilde{\mtx{P}}_2$ through a simpler approach since $\tilde{\mtx{P}}_2$ is also a reversible transition matrix. The next theorem shows that $\tilde{\mtx{P}}_2$ and $(\mtx{P}+\mtx{\Id})/2$ share the same nonzero eigenvalues with multiplicity accounted. As a result, the spectral gaps of $\tilde{\mtx{P}}_2$ can be obtained directly.
\begin{theorem}\label{thm:eigenvalue of tilde P2}
    Let $\mtx{P}$ be an irreducible reversible transition matrix on $\Omega$ with an invariant distribution $\mu$. Then $(\mtx{P}+\mtx{\Id})/2$ and $\tilde{\mtx{P}}_2$ share the same nonzero eigenvalues with multiplicity accounted. As a result, $\tilde{\mu}_2$ is the unique invariant distribution of $\tilde{\mtx{P}}_2$. Moreover, it holds that
    \[\eta_{a}\left(\tilde{\mtx{P}}_2\right)=\eta\left(\tilde{\mtx{P}}_2\right)=\frac{\eta(\mtx{P})}{2}>0.\]
\end{theorem}
\begin{proof}
    Similar to the proof for $\mtx{P}_2$, we decompose $\tilde{\mtx{P}}_2$ as the product of two transition matrices which represent step \ref{construction of tildeP2:step 1} and step \ref{construction of tildeP2:step 2} introduced in Subsection~\ref{subsubsec:construction of tildeP2}, respectively. To be specific, let $\tilde{\mtx{S}}$ denote the transition matrix of step \ref{construction of tildeP2:step 1} in which a given state $\tilde{\mtx{u}}\sim (u,v)$ will jump to $u$ or $v$ with equal probability $1/2$, i.e. for any $\tilde{\mtx{u}}\in \tilde{\Omega}^2$,$u\in \Omega$,
    \[\tilde{\mtx{S}}(\tilde{\mtx{u}},u)=
              \begin{cases} 1/2, &\tilde{\mtx{u}}\sim (u,v), \text{where } v\ne u, \\
              1, &\tilde{\mtx{u}}\sim  (u,u),\\
              0, &\text{otherwise}.\end{cases}\]
    Let $\tilde{\mtx{T}}$ denote the transition matrix of step \ref{construction of tildeP2:step 2} in which a given state $u$ will jump to the state $\tilde{\mtx{v}}\sim (u,v)$ according to the probability 
    \[\mathbb{P}\left\{\text{choose } v \right\}=p(u,v),\]
    i.e. for any $u\in \Omega$, $\tilde{\mtx{v}}\in \tilde{\Omega}^2$, 
    \[\tilde{\mtx{T}}(u,\tilde{\mtx{v}})=
    \begin{cases} p(u,v),&\tilde{\mtx{v}} \sim (u,v), v\in \Omega,  \\
              0, &\text{otherwise}.\end{cases}\]   
    Then we have $\tilde{\mtx{P}}_2=\tilde{\mtx{S}}\tilde{\mtx{T}}$. Similarly, for any $u,v\in \Omega$ such that $u\ne v$, we have
    \[\tilde{\mtx{T}}\tilde{\mtx{S}}(u,v)=\sum\limits_{\tilde{\mtx{u}}\in\tilde{\Omega}^2}\tilde{\mtx{T}}(u,\tilde{\mtx{u}})\tilde{\mtx{S}}(\tilde{\mtx{u}},v)=\frac{1}{2}p(u,v).\]
    On the other hand, direct calculation shows that
    \begin{align*}
        \tilde{\mtx{T}}\tilde{\mtx{S}}(u,u)&=\sum\limits_{\tilde{\mtx{u}}\in\tilde{\Omega}^2}\tilde{\mtx{T}}(u,\tilde{\mtx{u}})\tilde{\mtx{S}}(\tilde{\mtx{u}},u)=\sum\limits_{\tilde{\mtx{u}}\sim(u,v),v\in \Omega}\tilde{\mtx{T}}(u,\tilde{\mtx{u}})\tilde{\mtx{S}}(\tilde{\mtx{u}},u)\\
        &=\sum\limits_{v\ne u}\frac{1}{2}p(u,v)+p(u,u)=\frac{1}{2}+\frac{1}{2}p(u,u).
    \end{align*}
    Thus $\tilde{\mtx{T}}\tilde{\mtx{S}}=(\mtx{P}+\mtx{\Id})/2$, which implies that $(\mtx{P}+\mtx{\Id})/2$ and $\tilde{\mtx{P}}_2$ share the same nonzero eigenvalues with multiplicity accounted. As a direct corollary, by the irreducibility of $\mtx{P}$ we know 1 is a simple eigenvalue of $\mtx{P}$, and therefore 1 is also a simple eigenvalue of $\tilde{\mtx{P}}_2$, which establishes the uniqueness of the invariant distribution of $\tilde{\mtx{P}}_2$.
  
    Furthermore, by the reversibility of $\tilde{\mtx{P}}_2$ and the fact that all eigenvalues of $\tilde{\mtx{P}}_2$ are nonnegative, we deduce
    \[\lambda_a\left(\tilde{\mtx{P}}_2\right)= \lambda\left(\tilde{\mtx{P}}_2\right)= \lambda\left(\frac{\mtx{P}+\mtx{\Id}}{2}\right)=\frac{1+\lambda(\mtx{P})}{2},\]
    which implies
    \[\eta_a\left(\tilde{\mtx{P}}_2\right)= \eta\left(\tilde{\mtx{P}}_2\right)=1-\frac{1+\lambda(\mtx{P})}{2}= \frac{\eta(\mtx{P})}{2}.\]
    This completes the proof.
\end{proof}

\section{Non-asymptotic error bounds}\label{sec:error bounds}
With the spectral properties of $\mtx{P}_2$ and $\tilde{\mtx{P}}_2$ in hand, we can now prove the main results of this paper by establishing non-asymptotic error bounds for MLE and SCE, respectively.
\subsection{Tail bounds of MLE and SCE}\label{subsec:tail bound}
We will first derive tail bounds for the error of MLE and SCE using concentration inequalities for Markov chains.  Let us state the complete version of \cite[Theorem 2.4]{huang2026bernstein}, which is a Bernstein-type inequality for Markov sums that only requires a nonzero IP gap of the transition matrix.

\begin{theorem}\label{thm:scalar concentration}
    Let $\{u_k\}_{k\geq 1}$ be a Markov chain on $\Omega$ generated by a transition kernel $\mtx{A}$ with an invariant distribution $\mu$ and an IP gap $\eta_p(\mtx{A})>0$, and let $g:\Omega\rightarrow \R$ be a scalar function that satisfies $\mathbb{E}_{\mu}[g]=0$, $\mathbb{E}_{\mu}[g^2]\leq \sigma^2$, and $|g|\leq M$. It holds that for any $t>0$,
    \[\mathbb{P}\left\{\left|\frac{1}{n}\sum\limits_{j=1}^n g(u_j)\right|\geq t\right\}\leq 2\|\nu/\mu\|_{\infty}\exp\left(\frac{-n\,\eta_p(\mtx{A})\, t^2}{4M\sqrt{\left(2+6\eta_p(\mtx{A})\right)^2\sigma^2+t^2}}\right),\]
    where $\nu$ is the initial distribution of $u_1$.
\end{theorem}
In the above theorem, we use the symbol $\mtx{A}$ for an arbitrary Markov transition kernel, so that it will not be confused with the particular transition matrix $\mtx{P}$ in our context. Since $\eta_p(\mtx{A})\geq \eta_s(\mtx{A})\geq \eta_a(\mtx{A})$ and $\eta_p(\mtx{A})\geq \eta_{ps}(\mtx{A})/2$, this concentration inequality also holds if the transition admits a nonzero absolute spectral gap, a nonzero pseudo spectral gap, or a nonzero symmetric spectral gap. We then use Theorem \ref{thm:scalar concentration} to derive non-asymptotic tail bounds for MLE.
\begin{proof}[Proof of Theorem~\ref{thm:main results of natural counting}~(\ref{item:MLEmainresults convergence})(\ref{item:MLEmainresults tail bound})]
    Recall that in MLE, the corresponding matrix empirical mean $\mtx{R}_n$ can be reformulated as
    \[\mtx{R}_n=\frac{1}{n}\sum\limits_{k=1}^n\mtx{F}(\mtx{u}_k),\]
    where $\{\mtx{u}_k\}_{k\geq 1}$ is the Markov chain constructed in Subsection~\ref{subsubsec:construction of P2}, and $\mtx{F}=[\mathbb{1}_{(u,v)}]_{u,v\in \Omega}$ is the matrix-valued function whose $(u,v)$ entry is the indicator function $\mathbb{1}_{(u,v)}$ of the pair $(u,v)\in \Omega^2$. It is shown that 
    \[\E_{\mu_2}\left[\mtx{F}\right]=\mtx{D}_\mu\mtx{P}\]
    in Lemma \ref{lem:unbiased of F and tildeF}, and thus the almost surely convergence of $\mtx{R}_n$ is a direct corollary of the standard ergodic theorem for Markov chains.

    Next, for any $u,v\in \Omega$, we take $g=\mathbb{1}_{(u,v)}-\mu(u)p(u,v)$. Then it is easy to verify that $\mathbb{E}_{\mu_2}[g]=0$, $|g|\leq 1$, and $\Var_{\mu_2}[g]\leq \mu_2(u,v)=\mu(u)p(u,v)$. Recall that $\mtx{P}_2$ is the transition matrix of $\{\mtx{u}_k\}_{k\geq 1}$. We can therefore apply Theorem~\ref{thm:scalar concentration} with $\mtx{A}=\mtx{P}_2$, $M=1$, $\sigma^2\leq \mu_2(u,v)$ to deduce
    \[\mathbb{P}\left\{\left|r_n(u,v)-\mu(u)p(u,v)\right|\geq t\right\}\leq 2\|\nu_2/\mu_2\|_{\infty}\exp\left(\frac{-n\, \eta_p(\mtx{P}_2)\, t^2}{4\sqrt{\left(2+6\eta_p(\mtx{P}_2)\right)^2 \mu_2(u,v)+t^2}}\right),\]
    where $\nu_2$ is the distribution of $\mtx{u}_1$. Note that for any $\mtx{w}=(w_1,w_2)\in \Omega^2$, we have 
    \[\frac{\nu_2(\mtx{w})}{\mu_2(\mtx{w})}=\frac{\nu(w_1)p(w_1,w_2)}{\mu(w_1)p(w_1,w_2)}=\frac{\nu(w_1)}{\mu(w_1)} ,\]
    which implies $\|\nu_2/\mu_2\|_{\infty}=\|\nu/\mu\|_{\infty}$. Finally, we use the fact $\eta_p(\mtx{P}_2)\geq \eta_p(\mtx{P})/(1+\eta_p(\mtx{P}))$ to conclude the desired result.
\end{proof}
As for SCE, we choose to derive non-asymptotic tail bounds directly in the matrix form using matrix concentration inequalities. Recently,~\cite{neeman2024concentration} established Hoeffding and Bernstein-type matrix concentration inequalities (in $\ell_2$ operator norm) for sums of Markov dependent random matrices, under the assumption that the transition matrix admits a nonzero absolute spectral gap. We present a simplified version of~\cite[Theorem 2.2]{neeman2024concentration} as follows.
\begin{theorem}\label{thm:matrix concentration}
    Let $\{u_k\}_{k\geq 1}$ be a Markov chain on $\Omega$ generated by a transition kernel $\mtx{A}$ with an invariant distribution $\mu$ and an absolute spectral gap $\eta_a(\mtx{A})>0$, and let $\mtx{G}:\Omega\rightarrow \mathbb{H}_d$ be a matrix-valued function that satisfies $\mathbb{E}_{\mu}[\mtx{G}]=0$, $\|\mathbb{E}_{\mu}[\mtx{G}^2]\|\leq \sigma^2$, and $\|\mtx{G}\|\leq M$. It holds that for any $t>0$,
    \[\mathbb{P}\left\{\left\|\frac{1}{n}\sum\limits_{j=1}^n\mtx{G}(u_j)\right\|\geq t\right\}\leq 2\|\nu/\mu\|_{\infty}d^{2-\frac{\pi}{4}}\exp\left(\frac{-n\, \eta_a(\mtx{A})\,t^2}{32\pi^{-2} (2-\eta_a(\mtx{A}))\sigma^2+256\pi^{-3}Mt}\right),\]
    where $\nu$ is the initial distribution of $u_1$.
\end{theorem}
Since we have shown that $\eta_a(\tilde{\mtx{P}}_2)=\eta(\mtx{P})/2$, we can use Theorem~\ref{thm:matrix concentration} to establish non-asymptotic tail bounds for SCE.
\begin{proof}[Proof of Theorem~\ref{thm:main results of symmetric counting}~(\ref{item:SCEmainresults convergence})(\ref{item:SCEmainresults tail bound})]
    In SCE, the corresponding empirical mean $\mtx{H}_n$ is defined by
    \[\mtx{H}_n= \frac{1}{n}\sum\limits_{k=1}^n\tilde{\mtx{F}}(\tilde{\mtx{u}}_k),\]
    where $\{\tilde{\mtx{u}}_k\}_{k\geq 1}$ is a Markov chain generated by $\tilde{\mtx{P}}_2$, and $\tilde{\mtx{F}}:\tilde{\Omega}^2\rightarrow \mathbb{H}_{d}$ is the matrix-valued function on the symmetric path space $\tilde{\Omega}^2$ which satisfies
    \[\tilde{\mtx{F}}(\tilde{\mtx{w}})=\frac{1}{2}\left(\mtx{E}_{w_1w_2}+\mtx{E}_{w_2w_1}\right), \qquad \text{ if} \,\ \tilde{\mtx{w}}\sim (w_1,w_2).\]
    It is shown that 
    \[\E_{\tilde{\mu}_2}\left[\tilde{\mtx{F}}\right]=\mtx{D}_\mu\mtx{P}\]
    in Lemma \ref{lem:unbiased of F and tildeF}, and then the almost surely convergence of $\mtx{H}_n$ follows immediately.
 
    Next, we find that
    \[\left\|\tilde{\mtx{F}}-\mtx{D}_\mu\mtx{P}\right\|\leq \left\|\tilde{\mtx{F}}\right\|+\left\|\mtx{D}_\mu\mtx{P}\right\|\leq 2,\]
    and
    \[\left\|\mathbb{E}_{\tilde{\mu}_2}\left[\left(\tilde{\mtx{F}}-\mtx{D}_\mu\mtx{P}\right)^2\right]\right\|= \left\|\mathbb{E}_{\tilde{\mu}_2}\left[\tilde{\mtx{F}}^2\right]-\left(\mtx{D}_\mu\mtx{P}\right)^2\right\|\leq \left\|\mathbb{E}_{\tilde{\mu}_2}\left[\tilde{\mtx{F}}^2\right]\right\|+\left\|\mtx{D}_\mu\mtx{P}\right\|^2\leq 2.\]
    We then apply Theorem \ref{thm:matrix concentration} with $\mtx{A}=\tilde{\mtx{P}}_2$, $\mtx{G}=\tilde{\mtx{F}}-\mtx{D}_\mu\mtx{P}$, $M=\sigma^2=2$ to deduce that
    \[\mathbb{P}\left\{\left\|\mtx{H}_n-\mtx{D}_{\mu}\mtx{P}\right\|\geq t\right\}\leq 2\|\tilde{\nu}_2/\tilde{\mu}_2\|_{\infty}d^{2-\frac{\pi}{4}}\exp\left(\frac{-n\,\eta_a(\tilde{\mtx{P}}_2)\,t^2}{64\pi^{-2} (2-\eta_a(\tilde{\mtx{P}}_2))+512\pi^{-3}t}\right),\]
    where $\tilde{\nu}_2$ is the distribution of $\tilde{\mtx{u}}_1$. According to the generation of $\tilde{\mtx{u}}_1$, for any $\tilde{\mtx{w}}\sim(w_1,w_2)$, we can compute that
    \[\frac{\tilde{\nu}_2\left(\tilde{\mtx{w}}\right)}{\tilde{\mu}_2\left(\tilde{\mtx{w}}\right)}=\frac{\nu(w_1)p(w_1,w_2)+\nu(w_2)p(w_2,w_1)}{\mu(w_1)p(w_1,w_2)+\mu(w_2)p(w_2,w_1)}=\frac{1}{2}\left(\frac{\nu(w_1)}{\mu(w_1)}+\frac{\nu(w_2)}{\mu(w_2)}\right).\]
    Therefore, we have $\|\tilde{\nu}_2/\tilde{\mu}_2\|_{\infty}=\|\nu/\mu\|_{\infty}$. Finally, we use the fact that $\eta_a(\mtx{\tilde{P}_2})=\eta(\mtx{P})/2$ to conclude the proof.
\end{proof}

\subsection{Dimension-free expectation bound}
In this subsection, we derive dimension-free expectation bounds for the errors of MLE and SCE and complete the proofs of Theorem~\ref{thm:main results of natural counting} and Theorem~\ref{thm:main results of symmetric counting}.

\begin{proof}[Proof of Theorem~\ref{thm:main results of natural counting} (\ref{item:MLEmainresults MSE})]
    First, let $\nu$ be the initial distribution of $u_1$, and let $\nu_2$ be the initial distribution induced by $\nu$. Notice that
    \begin{align*}
        \mathbb{E}\left[\left\|\mtx{R}_n-\mtx{D}_{\mu}\mtx{P}\right\|_{\mathrm{F}}^2 \middle| u_1\sim \nu\right]&= \mathbb{E}\left[\left\|\mtx{R}_n-\mtx{D}_{\mu}\mtx{P}\right\|_{\mathrm{F}}^2 \middle| \mtx{u}_1\sim \nu_2\right]\\   &\leq \|\nu_2/\mu_2\|_{\infty}\mathbb{E}\left[\left\|\mtx{R}_n-\mtx{D}_{\mu}\mtx{P}\right\|_{\mathrm{F}}^2 \middle| \mtx{u}_1\sim \mu_2\right]\\ &  = \|\nu/\mu\|_{\infty}\mathbb{E}\left[\left\|\mtx{R}_n-\mtx{D}_{\mu}\mtx{P}\right\|_{\mathrm{F}}^2 \middle| \mtx{u}_1\sim \mu_2\right].
    \end{align*}
    Therefore, we only need to bound $\mathbb{E}[\|\mtx{R}_n-\mtx{D}_{\mu}\mtx{P}\|_{\mathrm{F}}^2]$ under the condition that $\mtx{u}_1\sim \mu_2$. For any $\mtx{v}\in \Omega^2$, let $\bar{\mathbb{1}}_{\mtx{v}}$ denote the projection of the indicator function $\mathbb{1}_{\mtx{v}}$ onto the mean-zero subspace $L_{2,\mu_2}^0$, i.e.
    \[\bar{\mathbb{1}}_{\mtx{v}}(\mtx{w})=\mathbb{1}_{\mtx{v}}(\mtx{w})-\mu_2(\mtx{v}),\quad \text{for any}\,\ \mtx{w}\in \Omega^2.\]
    We then use the definition of $\mtx{R}_n$ to compute that
    \begin{align*}
        \mathbb{E}\left[\left\|\mtx{R}_n-\mtx{D}_{\mu}\mtx{P}\right\|_{\mathrm{F}}^2\right]
        &=\mathbb{E}\left[\sum\limits_{\mtx{v}\in\Omega^2}\left(\frac{1}{n}\sum\limits_{k=1}^n \bar{\mathbb{1}}_{\mtx{v}}\left(\mtx{u}_k\right)\right)^2\middle|\mtx{u}_1\sim \mu_2\right]\\
        &=\frac{1}{n^2}\sum\limits_{\mtx{v}\in\Omega^2}\sum\limits_{1\leq i,j\leq n}\mathbb{E}\left[\bar{\mathbb{1}}_{\mtx{v}}\left(\mtx{u}_i\right)\cdot\bar{\mathbb{1}}_{\mtx{v}}\left(\mtx{u}_j\right)\middle|\mtx{u}_1\sim \mu_2\right]\\
        &=\frac{1}{n^2}\sum\limits_{\mtx{v}\in\Omega^2}\left(n\left\|\bar{\mathbb{1}}_{\mtx{v}}\right\|^2_{\mu_2}+2\sum\limits_{1\leq i<j\leq n}\mathbb{E}\left[\bar{\mathbb{1}}_{\mtx{v}}\left(\mtx{u}_i\right)\cdot\mtx{P}_2^{j-i}\bar{\mathbb{1}}_{\mtx{v}}\left(\mtx{u}_i\right)\middle|\mtx{u}_i\sim \mu_2\right]\right)\\
        &=\frac{1}{n^2}\sum\limits_{\mtx{v}\in\Omega^2}\left(n\left\|\bar{\mathbb{1}}_{\mtx{v}}\right\|^2_{\mu_2}+2\sum\limits_{k=1}^{n-1}(n-k)\left\langle\bar{\mathbb{1}}_{\mtx{v}},\mtx{P}_2^k\bar{\mathbb{1}}_{\mtx{v}}\right\rangle_{\mu_2}\right)\\
        &= \frac{1}{n}\left(\sum\limits_{\mtx{v}\in \Omega^2}\left\|\bar{\mathbb{1}}_{\mtx{v}}\right\|^2_{\mu_2}\right)+\frac{2}{n^2}\sum\limits_{\mtx{v}\in \Omega^2}\sum\limits_{k=1}^{n-1}(n-k)\left\langle\bar{\mathbb{1}}_{\mtx{v}},\left(\mtx{P}_2^k-\mtx{\Pi}_2\right)\bar{\mathbb{1}}_{\mtx{v}}\right\rangle_{\mu_2}\\
        &\leq \left(\frac{1}{n}+\frac{2}{n^2}\left\|\sum\limits_{k=1}^{n-1}(n-k)\left(\mtx{P}_2^k-\mtx{\Pi}_2 \right) \right\|_{\mu_2}\right) \left(\sum\limits_{\mtx{v}\in \Omega^2}\left\|\bar{\mathbb{1}}_{\mtx{v}}\right\|^2_{\mu_2}\right),
    \end{align*}
    where $\mtx{\Pi}_2:=\mtx{1}_{d^2}\mu_2^{\mathrm{T}}$ is the projection operator onto $\mtx{1}_{d^2}$ under the $L_{2,\mu_2}$ inner product.
    Since
    \[\sum\limits_{\mtx{v}\in \Omega^2}\left\|\bar{\mathbb{1}}_{\mtx{v}}\right\|^2_{\mu_2}\leq\sum\limits_{\mtx{v}\in \Omega^2}\left\|\mathbb{1}_{\mtx{v}}\right\|^2_{\mu_2}=\sum\limits_{\mtx{v}\in \Omega^2} \mu_2(\mtx{v})=1,\]
    we only need to bound the $L_{2,\mu_2}$ operator norm of $\sum_{k=1}^{n-1}(n-k)(\mtx{P}_2^k-\mtx{\Pi}_2)$. Let $\mtx{S}$, $\mtx{T}$ be defined in Theorem~\ref{thm:eigenvalue of P_2} which satisfy $\mtx{P}_2=\mtx{S}\mtx{T}$, $\mtx{P}=\mtx{T}\mtx{S}$, and let $\mtx{S}_1=\mtx{D}_{\mu_2}^{1/2}{\mtx{S}}\mtx{D}_{\mu}^{-1/2}$, $\mtx{T}_1=\mtx{D}_{\mu}^{1/2}{\mtx{T}}\mtx{D}_{\mu_2}^{-1/2}$. One can verify that $\mtx{S}\mtx{1}_d=\mtx{1}_{d^2}$ and $\mu^{\mathrm{T}}\mtx{T}=\mu_2^{\mathrm{T}}$, and thus
    \[\mtx{\Pi}_2=\mtx{1}_{d^2}\mu_2^{\mathrm{T}}=\mtx{S}\mtx{1}_d\mu^{\mathrm{T}}\mtx{T}=\mtx{S}\mtx{\Pi}\mtx{T},\]
    where $\mtx{\Pi}$ is the projection onto $\mtx{1}_d$ defined in Subsection~\ref{subsec:Fundamentals}. One can verify that $\mtx{P}^k-\mtx{\Pi}=(\mtx{P}-\mtx{\Pi})^k=(\mtx{P}-\mtx{\Pi})^k(\mtx{\Id}-\mtx{\Pi}) $ holds for any $k\in \N_+$. As a result, for any integer $k\geq 2$, we have 
    \[\mtx{P}_2^k-\mtx{\Pi}_2=\left(\mtx{S}\mtx{T}\right)^k-\mtx{S}\mtx{\Pi}\mtx{T}=\mtx{S}\left(\left(\mtx{T}\mtx{S}\right)^{k-1}-\mtx{\Pi}\right)\mtx{T}=\mtx{S}\left(\mtx{P}-\mtx{\Pi}\right)^{k-1}(\mtx{\Id}-\mtx{\Pi})\mtx{T}.\]
    In particular, $\mtx{P}_2-\mtx{\Pi}_2=\mtx{S}(\mtx{\Id}-\mtx{\Pi})\mtx{T}$. For convenience, we define a polynomial 
    \[\phi(x):=\frac{2}{n^2}\sum\limits_{k=1}^{n-1}(n-k)x^{k-1}.\]
    Then
    \begin{align*}
        \frac{2}{n^2}\left\|\sum\limits_{k=1}^{n-1}(n-k)\left(\mtx{P}_2^k-\mtx{\Pi}_2 \right) \right\|_{\mu_2}&= \frac{2}{n^2}\left\|\mtx{S}\left((n-1)\mtx{\Id}+\sum\limits_{k=2}^{n-1}(n-k)\left(\mtx{P}-\mtx{\Pi}\right)^{k-1}\right)(\mtx{\Id}-\mtx{\Pi})\mtx{T}\right\|_{\mu_2}\\
        &= \left\| \mtx{S}\phi\left(\mtx{P}-\mtx{\Pi}\right)(\mtx{\Id}-\mtx{\Pi})\mtx{T} \right\|_{\mu_2}\\
        &= \left\| \mtx{S}_1 \mtx{D}_{\mu}^{1/2}\phi\left(\mtx{P}-\mtx{\Pi}\right)(\mtx{\Id}-\mtx{\Pi})\mtx{D}_{\mu}^{-1/2}\mtx{T}_1 \right\|\\
        &\leq \|\mtx{S}_1\|\|\mtx{T}_1\|\left\|\mtx{D}_{\mu}^{1/2}\phi\left(\mtx{P}-\mtx{\Pi}\right)(\mtx{\Id}-\mtx{\Pi})\mtx{D}_{\mu}^{-1/2}\right\|\\
        &=\left\|\phi\left(\mtx{P}-\mtx{\Pi}\right)(\mtx{\Id}-\mtx{\Pi})\right\|_{\mu},
    \end{align*}
    where the last equality is due to the fact that $\mtx{S}_1^{\mathrm{T}}\mtx{S}_1=\mtx{T}_1\mtx{T}_1^{\mathrm{T}}=\mtx{\Id}$ (Theorem \ref{thm:eigenvalue of P_2}). Our goal is to bound $\|\phi(\mtx{P}-\mtx{\Pi})(\mtx{\Id}-\mtx{\Pi})\|_{\mu}$. To achieve this, consider an arbitrary $h\in L_{2,\mu}$, and let $\bar{h}:=(\mtx{\Id}-\mtx{\Pi})h$. If $\bar{h} =\mtx{0}$, then $\phi(\mtx{P}-\mtx{\Pi})(\mtx{\Id}-\mtx{\Pi})h=\mtx{0}$. Otherwise,
    \[\frac{\left\|\phi(\mtx{P}-\mtx{\Pi})(\mtx{\Id}-\mtx{\Pi})h\right\|_{\mu}}{\left\|h\right\|_{\mu}}\leq  \frac{\left\|\phi(\mtx{P}-\mtx{\Pi})(\mtx{\Id}-\mtx{\Pi})h\right\|_{\mu}}{\left\|\bar{h}\right\|_{\mu}}=  \frac{\left\|\phi(\mtx{P}-\mtx{\Pi})\bar{h}\right\|_{\mu}}{\left\|\bar{h}\right\|_{\mu}}= \frac{\left\|\phi(\mtx{P})\bar{h}\right\|_{\mu}}{\left\|\bar{h}\right\|_{\mu}} .\]
    As a result,
    \[\left\|\phi(\mtx{P}-\mtx{\Pi})(\mtx{\Id}-\mtx{\Pi})\right\|_{\mu} =\sup\limits_{h\in L_{2,\mu},\  h\ne \mtx{0}}\frac{\left\|\phi(\mtx{P}-\mtx{\Pi})(\mtx{\Id}-\mtx{\Pi})h\right\|_{\mu}}{\left\|h\right\|_{\mu}}=\sup\limits_{h\in L_{2,\mu}^0,\ h\ne \mtx{0}}\frac{\left\|\phi(\mtx{P})h\right\|_{\mu}}{\left\|h\right\|_{\mu}}.\]
    By the irreducibility of $\mtx{P}$, we have $\operatorname{rank}(\mtx{\Id}-\mtx{P})=d-1$. We also know that for any $h\in L_{2,\mu}$, $\E_\mu[(\mtx{\Id}-\mtx{P})h]=0$, and that $(\mtx{\Id}-\mtx{P})\mtx{1}=0$. Thus, $\mtx{\Id}-\mtx{P}$ is a bijection from $L_{2,\mu}^0$ to $L_{2,\mu}^0$. Hence, for any $h\in L_{2,\mu}^0$, there exists a function $g\in L_{2,\mu}^0$ such that $(\mtx{\Id}-\mtx{P})g=h$. For such $g$, by the definition of the IP gap $\eta_p(\mtx{P})$, we have $\|g\|_{\mu}\leq \|h\|_{\mu}/\eta_p(\mtx{P})$. Therefore,
    \[\left\|\phi(\mtx{P})h\right\|_{\mu}=\left\|\phi(\mtx{P})(\mtx{\Id}-\mtx{P})g\right\|_{\mu}\leq \left\|\phi(\mtx{P})(\mtx{\Id}-\mtx{P})\right\|_{\mu} \left\|g\right\|_{\mu}\leq \frac{\left\|\phi(\mtx{P})(\mtx{\Id}-\mtx{P})\right\|_{\mu}}{\eta_p(\mtx{P})}\|h\|_{\mu}.\]
    To continue, notice that
    \[\left\|\phi(\mtx{P})(\mtx{\Id}-\mtx{P})\right\|_{\mu}=\frac{2}{n^2}\left\|(n-1)\mtx{\Id} -\sum\limits_{k=1}^{n-1}\mtx{P}^k  \right\|_{\mu}\leq \frac{2}{n^2}\left((n-1)+\sum\limits_{k=1}^{n-1}\left\|\mtx{P}^k\right\|_{\mu}\right)\leq\frac{4}{n}.\]
    We thus deduce
    \[\left\|\phi(\mtx{P})h\right\|_{\mu}\leq \frac{4}{n\eta_p(\mtx{P})}\left\|h\right\|_{\mu},\quad \text{for any}\,\ h\in L_{2,\mu}^0,\]
    which implies $\|\phi(\mtx{P}-\mtx{\Pi})(\mtx{\Id}-\mtx{\Pi})\|_{\mu}\leq 4/(n\eta_p(\mtx{P}))$. As a result,
    \begin{equation}\label{eqt:bound1}
        \mathbb{E}\left[\left\|\mtx{R}_n-\mtx{D}_{\mu}\mtx{P}\right\|_{\mathrm{F}}^2\right]\leq \left(\frac{1}{n}+\left\|\phi(\mtx{P}-\mtx{\Pi})(\mtx{\Id}-\mtx{\Pi})\right\|_{\mu}\right)\left(\sum\limits_{\mtx{v}\in \Omega^2}\left\|\bar{\mathbb{1}}_{\mtx{v}}\right\|^2_{\mu_2}\right)\leq \frac{1}{n}+\frac{4}{n\eta_p(\mtx{P})}.
    \end{equation}
    which is the claimed dimension-free expectation bound.

    If $\mtx{P}$ admits a nonzero absolute spectral gap $\eta_a(\mtx{P})=1-\lambda_a(\mtx{P})$, we can bound $\|\phi(\mtx{P}-\mtx{\Pi})\|_{\mu}$ alternatively using the fact that
    \begin{equation}\label{eqt:bound of g}
        \left|\phi(x)\right|=\frac{2}{n^2}\left|\sum\limits_{k=1}^{n-1}(n-k)x^{k-1}\right|=\frac{2}{n^2}\left(\frac{n}{1-x}+\frac{x^{n}-1}{(1-x)^2}\right)\leq \frac{2}{n(1-x)}, \quad \text{for any}\,\  x\in (-1,1).
    \end{equation}
    Since all coefficients of $\phi$ are nonnegative, we have
    \begin{equation}\label{eqt:spectral estimate for g}
        \left\|\phi(\mtx{P}-\mtx{\Pi})(\mtx{\Id}-\mtx{\Pi})\right\|_{\mu} \leq\left\| \phi\left(\mtx{P}-\mtx{\Pi}\right)\right\|_{\mu}\leq \phi\left( \left\|\mtx{P}-\mtx{\Pi}\right\|_{\mu}\right)=\phi\left(\lambda_a(\mtx{P})\right)\leq \frac{2}{n\eta_a(\mtx{P})}.
    \end{equation}
    It follows that,
    \begin{equation}\label{eqt:bound2}
        \mathbb{E}\left[\left\|\mtx{R}_n-\mtx{D}_{\mu}\mtx{P}\right\|_{\mathrm{F}}^2 \right]\leq \frac{2+\eta_a(\mtx{P})}{n\eta_a(\mtx{P})}.
    \end{equation}

    Moreover, when $\mtx{P}$ is reversible, $\mtx{P}-\mtx{\Pi}$ is self-adjoint under the $L_{2,\mu}$ inner product, and $\|\phi(\mtx{P}-\mtx{\Pi})\|_{\mu}$ is the largest singular value of the symmetric matrix $\phi(\mtx{D}_{\mu}^{1/2}(\mtx{P}-\mtx{\Pi})\mtx{D}_{\mu}^{-1/2})$. In this case,~\eqref{eqt:spectral estimate for g} can be improved to
    \begin{align*}
        \left\|\phi\left(\mtx{P}-\mtx{\Pi}\right)\right\|_{\mu}&=\left\|\phi\left(\mtx{D}_{\mu}^{1/2}(\mtx{P}-\mtx{\Pi})\mtx{D}_{\mu}^{-1/2}\right)\right\|\\ & =\sup\limits_{\lambda\in \operatorname{spec}\left(\mtx{P}-\mtx{\Pi}\right) }|\phi(\lambda)|\leq \sup\limits_{\lambda\in \operatorname{spec}\left(\mtx{P}-\mtx{\Pi}\right)}\frac{2}{n(1-\lambda)}= \frac{2}{n\eta(\mtx{P})},
    \end{align*}
    which implies 
    \[\mathbb{E}\left[\left\|\mtx{R}_n-\mtx{D}_{\mu}\mtx{P}\right\|_{\mathrm{F}}^2\right]\leq \frac{2+\eta(\mtx{P})}{n\eta(\mtx{P})}.\]
    This completes the proof of Theorem~\ref{thm:main results of natural counting}.
\end{proof}

\begin{remark}
    Whether the error depends on the dimension shall rely on the choice of the error metric. The reason we can obtain dimension free expectation bound is that we measure the error using the Frobenius norm, which is a natural choice for matrices. If one instead uses the TV norm to measure the error, a dependence on the dimension typically arises. This is why in \cite{wolfer2021statistical}, the error bound depends on the quantity $\|\mtx{D}_{\mu}\mtx{P}\|_{1/2}=(\sum_{u,v\in \Omega} (\mu(u)p(u,v))^{1/2})^2$, which may depend on the dimension $d$ and can be bounded by $d^2$.
    
    To better illustrate this point, consider estimating a discrete distribution on a state space of size $d$ by the empirical distribution of a large amount of i.i.d. samples. Then one can see that the $L_2$ error between the empirical distribution and the ground truth can be bounded independently of $d$. In contrast, the error in the TV norm scales like $d^{1/2}$ in the worst case.
\end{remark}

\begin{remark}
    It is worth noting that Theorem \ref{thm:main results of natural counting} can be extended to Markov chains on a countable state space, provided that the transition operator $\mtx{P}$ admits a nonzero IP gap $\eta_p(\mtx{P}) > 0$. 
    The only essential modification appears in the proof of Theorem \ref{thm:main results of natural counting}(\ref{item:MLEmainresults MSE}) where we need to show that the operator $\mtx{\Id} - \mtx{P}$ is invertible on $L_{2,\mu}^0$. In the finite-dimensional case, this is straightforward, since $\mtx{\Id} - \mtx{P}$ is one to one. In the countable (infinite-dimensional) case, we can prove this by contradiction. Suppose $\mtx{\Id} - \mtx{P}$ is not invertible on $L_{2,\mu}^0$. Then there exists a nonzero function $g \in L_{2,\mu}^0$ such that for all $h \in L_{2,\mu}^0$, $\la g, (\mtx{\Id}-\mtx{P})h\ra_{\mu}=0$. In particular, by taking $h=g$, we deduce that $\|g\|_{\mu}^2=\la g, \mtx{P}g\ra_{\mu}$. On the other hand, we have
    \[\la g, \mtx{P}g\ra_{\mu} \leq \|g\|_{\mu} \|\mtx{P}g\|_{\mu} \leq \|g\|_{\mu}^2,\] so equality must hold throughout, which implies $(\mtx{\Id}-\mtx{P}) g=0$. This contradicts the fact that $\eta_p(\mtx{P})>0$.
\end{remark}

\begin{remark}
    The estimate in \eqref{eqt:bound1} can be viewed as a matrix version of \cite[Theorem 1.2]{chatterjee2025spectral}.In fact, one can apply \cite[Theorem 1.2]{chatterjee2025spectral} to $\{\mtx{u}_k\}_{k\geq 1}$ to obtain that 
    \[\mathbb{E}\left[\left\|\mtx{R}_n-\mtx{D}_{\mu}\mtx{P}\right\|_{\mathrm{F}}^2\right]\lesssim \frac{1}{n\eta_p(\mtx{P}_2)}\lesssim  \frac{1}{n\eta_p(\mtx{P})}.\] The second inequality above owes to $\eta_p(\mtx{P}_2)\geq \eta_p(\mtx{P})/(1+\eta_p(\mtx{P}))$ in Theorem \ref{thm:eigenvalue of P_2}. In contrast, the proof we present above is more straightforward in the sense that it controls the expected error in the Frobenius norm directly by $1/(n\eta_p(\mtx{P}))$ without explicitly establishing the relation between $\eta_p(\mtx{P}_2) $ and $\eta_p(\mtx{P})$.
\end{remark}

\begin{remark}
    Compared with \eqref{eqt:bound1}, the error bound in \eqref{eqt:bound2} has a better constant. Therefore, \eqref{eqt:bound2} could be sharper than \eqref{eqt:bound1} when $\eta_a(\mtx{P})\geq \eta_p(\mtx{P})/2$. However, we cannot guarantee that $\mtx{P}$ admits a nonzero absolute spectral gap $\eta_a(\mtx{P})$. In contrast, $\eta_p(\mtx{P})>0$ holds for any irreducible transition matrix $\mtx{P}$. From this point of view, \eqref{eqt:bound1} is a more robust result.
\end{remark}

\begin{proof}[Proof of Theorem~\ref{thm:main results of symmetric counting}~(\ref{item:SCEmainresults MSE})]
    Let $\tilde{\nu}_2$ be the initial distribution induced by $\nu$. Since
    \begin{align*}
        \mathbb{E}\left[\left\|\mtx{H}_n-\mtx{D}_{\mu}\mtx{P}\right\|_{\mathrm{F}}^2\middle| \tilde{\mtx{u}}_1\sim {\tilde{\nu}_2}\right]&\leq 
        \|\tilde{\nu}_2/\tilde{\mu}_2\|_{\infty} \mathbb{E}\left[\left\|\mtx{H}_n-\mtx{D}_{\mu}\mtx{P}\right\|_{\mathrm{F}}^2\middle| \tilde{\mtx{u}}_1\sim {\tilde{\mu}_2}\right]\\ &= \|{\nu}/{\mu}\|_{\infty} \mathbb{E}\left[\left\|\mtx{H}_n-\mtx{D}_{\mu}\mtx{P}\right\|_{\mathrm{F}}^2\middle| \tilde{\mtx{u}}_1\sim {\tilde{\mu}_2}\right],
    \end{align*}
    it suffices to bound $\mathbb{E}[\|\mtx{H}_n-\mtx{D}_{\mu}\mtx{P}\|_{\mathrm{F}}^2]$ under the condition that $\tilde{\mtx{u}}_1\sim {\tilde{\mu}_2}$.
    For any $\tilde{\mtx{v}}\in \tilde{\Omega}^2$, we use $\bar{\mathbb{1}}_{\tilde{\mtx{v}}}$ to denote the projection of the indicator function $\mathbb{1}_{\tilde{\mtx{v}}}$ onto the mean-zero subspace $L_{2,\tilde{\mu}_2}^0$, i.e.
    \[\bar{\mathbb{1}}_{\tilde{\mtx{v}}}({\tilde{\mtx{w}}})=\mathbb{1}_{\tilde{\mtx{v}}}(\tilde{\mtx{w}})-\tilde{\mu}_2(\tilde{\mtx{v}}),\quad \text{for any}\,\ \tilde{\mtx{w}}\in \tilde{\Omega}^2.\]
    By the definition of $\mtx{H}_n=[h_n(u,v)]_{u,v\in \Omega}$, we have
    \[h_n(u,v)-\mu(u)p(u,v)=\begin{cases}
        \sum_{k=1}^n\bar{\mathbb{1}}_{\widetilde{(u,v)}}(\tilde{\mtx{u}}_k)/(2n)\quad &\text{if }\,\ u\ne v,\\
        \sum_{k=1}^n\bar{\mathbb{1}}_{\widetilde{(u,v)}}(\tilde{\mtx{u}}_k)/n\quad &\text{if }\,\ u=v.
    \end{cases}\]
    Next, we define a polynomial $\tilde{\phi}$ as
    \[\tilde{\phi}(x):= \frac{1}{n}+\frac{2}{n^2}\sum\limits_{k=1}^{n-1}(n-k)x^k,\]
    and then compute that
    \begin{align*}
        & \quad \,\ \mathbb{E}\left[\left\|\mtx{H}_n-\mtx{D}_{\mu}\mtx{P}\right\|_{\mathrm{F}}^2\right]\\
        &= \mathbb{E}\left[\sum\limits_{\substack{u,v\in \Omega\\u\ne v}} \left(\frac{1}{2n}\sum\limits_{k=1}^n \left({\bar{\mathbb{1}}_{\widetilde{(u,v)}}\left(\tilde{\mtx{u}}_k\right)}\right)\right)^2 +\sum\limits_{\substack{u\in \Omega}}\left(\frac{1}{n}\sum\limits_{k=1}^n \left({\bar{\mathbb{1}}_{\widetilde{(u,u)}}\left(\tilde{\mtx{u}}_k\right)}\right)\right)^2      \middle| \tilde{\mtx{u}}_1\sim {\tilde{\mu}_2}\right]    \\
        &\leq \mathbb{E}\left[ \sum\limits_{\tilde{\mtx{v}}\in\tilde{\Omega}^2}\left(\frac{1}{n}\sum\limits_{k=1}^n \left({\bar{\mathbb{1}}_{\tilde{\mtx{v}}}\left(\tilde{\mtx{u}}_k\right)}\right)\right)^2 \middle| \tilde{\mtx{u}}_1\sim {\tilde{\mu}_2}\right]\\
        &= \frac{1}{n^2}\sum\limits_{\tilde{\mtx{v}}\in\tilde{\Omega}^2}\sum\limits_{1\leq i,j\leq n}\mathbb{E}\left[\bar{\mathbb{1}}_{\tilde{\mtx{v}}}\left(\tilde{\mtx{u}}_i\right)\cdot\bar{\mathbb{1}}_{\tilde{\mtx{v}}}\left(\tilde{\mtx{u}}_j\right)\middle| \tilde{\mtx{u}}_1\sim {\tilde{\mu}_2}\right]\\
        &=\sum\limits_{\tilde{\mtx{v}}\in\tilde{\Omega}^2} \left(\frac{1}{n}\left\langle\bar{\mathbb{1}}_{\tilde{\mtx{v}}},\bar{\mathbb{1}}_{\tilde{\mtx{v}}}\right\rangle_{\tilde{\mu}_2}+\frac{2}{n^2} \sum\limits_{k=1}^{n-1} (n-k)\left\langle \left(\tilde{\mtx{P}}_2^k-\tilde{\mtx{\Pi}}_2\right)\bar{\mathbb{1}}_{\tilde{\mtx{v}}},\bar{\mathbb{1}}_{\tilde{\mtx{v}}}\right\rangle_{\tilde{\mu}_2}\right)\\
        &\leq \left\| \frac{\mtx{\Id}}{n}+\frac{2}{n^2}\sum\limits_{k=1}^{n-1} (n-k) \left(\tilde{\mtx{P}}_2^k-\tilde{\mtx{\Pi}}_2\right) \right\|_{\tilde{\mu}_2}\left(\sum\limits_{\tilde{\mtx{v}}\in\tilde{\Omega}^2} \left\| \bar{\mathbb{1}}_{\tilde{\mtx{v}}}\right\|_{\tilde{\mu}_2}^2 \right)\\
        &\leq \left\| \tilde{\phi} \left(\tilde{\mtx{P}}_2-\tilde{\mtx{\Pi}}_2\right) \right\|_{\tilde{\mu}_2},
    \end{align*}
    where $\tilde{\mtx{\Pi}}_2$ is the projection operator onto $\mtx{1}_{d(d+1)/2}$ under the $L_{2,\tilde{\mu}_2}$ inner product. For the last inequality, we use the facts that $\tilde{\mtx{P}}_2^k-\tilde{\mtx{\Pi}}_2=(\tilde{\mtx{P}}_2-\tilde{\mtx{\Pi}}_2)^k$ and $\sum_{\tilde{\mtx{v}}\in\tilde{\Omega}^2} \| \bar{\mathbb{1}}_{\tilde{\mtx{v}}}\|_{\tilde{\mu}_2}^2 \leq 1$. Notice that for any $x\in[0,1)$,
    \[\tilde{\phi}(x)=\frac{1}{n}+\frac{2}{n^2}\sum\limits_{k=1}^{n-1}(n-k) x^k=\frac{1}{n^2}\left(\frac{n(1+x)}{1-x}+\frac{2x^{n+1}-2x}{\left(1-x\right)^2}\right)\leq \frac{1+x}{n(1-x)}.\]
    Therefore,
     \begin{align*}
        \left\|\tilde{\phi} \left(\tilde{\mtx{P}}_2-\tilde{\mtx{\Pi}}_2\right) \right\|_{\tilde{\mu}_2}&\leq  \tilde{\phi} \left(\left\|\tilde{\mtx{P}}_2-\tilde{\mtx{\Pi}}_2\right\|_{\tilde{\mu}_2}\right)=\tilde{\phi}\left(\lambda_a\left(\tilde{\mtx{P}}_2\right)\right)\\
        &\leq \frac{1+\lambda_a\left(\tilde{\mtx{P}}_2\right)}{n\left(1-\lambda_a\left(\tilde{\mtx{P}}_2\right)\right)}=\frac{2-\eta_a\left(\tilde{\mtx{P}}_2\right)}{n\eta_a\left(\tilde{\mtx{P}}_2\right)}=\frac{4-\eta(\mtx{P})}{n\eta\left(\mtx{P}\right)},
    \end{align*}
    where the last equality is due to the fact that $\eta_a(\tilde{\mtx{P}}_2)=\eta(\mtx{P})/2$ (Theorem \ref{thm:eigenvalue of tilde P2}). As a result, we obtain
    \[\mathbb{E}\left[\left\|\mtx{H}_n-\mtx{D}_{\mu}\mtx{P}\right\|_{\mathrm{F}}^2\right]\leq \frac{4-\eta(\mtx{P})}{n\eta\left(\mtx{P}\right)},\]
    which completes the proof.
\end{proof}
\begin{remark}
    We can also use the tail bounds proved in Subsection~\ref{subsec:tail bound} to derive expectation bounds for SCE. In general, let $X$ be a random variable such that
    \[\mathbb{P}\left\{|X|\geq t\right\}\leq c_1\econst^{-c_2t^2}, \quad \text{for any}\,\ t>0,\]
    where $c_1\geq 1, c_2>0$ are two constants. Then $\mathbb{E}[X^2]$ can be bounded by integrating the tail bound:
    \begin{equation}\label{eqt:lem:expectation via tail bound}
        \mathbb{E}\left[X^2\right]=\int_{0}^{+\infty}\mathbb{P}\left\{|X|^2> t\right\}\diff t\leq \int_{0}^{+\infty}\min\left\{1,c_1\econst^{-c_2t}\right\}\diff t= \frac{\log\left(c_1\econst\right)}{c_2}.
    \end{equation}
    Thus we can deduce $\E[\|\mtx{H}_n-\mtx{D}_{\mu}\mtx{P}\|^2]=\mathcal{O}(\log(d)/(n\eta(\mtx{P})))$ from the tail bounds of SCE. However, from the proofs given in this subsection, we know the error expectations of the two methods are both independent of the dimension $d$.
\end{remark}
   
\section{Numerical results}\label{sec:numerical results} 
In this section, we present a numerical study to verify and visualize our theoretical results. In order to compare MLE and SCE, we take $\mtx{P}$ to be a reversible transition matrix. Besides, we adopt the following strategy to adjust the spectral gap of $\mtx{P}$: for $0<\eta<1$ and a reversible transition matrix $\mtx{P}$ with an invariant distribution $\mu$, we take
\[\mtx{P}_{\eta}=\begin{cases} \frac{\eta}{1-\lambda_2}\mtx{P}+\left(1-\frac{\eta}{1-\lambda_2}\right)\mtx{\Id},&\eta\leq 1-\lambda_2,\\\frac{1-\eta}{\lambda_2}\mtx{P}+\left(1-\frac{1-\eta}{\lambda_2}\right)\mtx{1}\mu^{\mathrm{T}},&\eta> 1-\lambda_2,
    \end{cases}\]
where $\lambda_2$ is the second largest eigenvalue of $\mtx{P}$. Then it is not hard to see that $\mtx{P}_{\eta}$ is a reversible transition matrix with a spectral gap $\eta$. For a given transition matrix $\mtx{P}$, we run MLE and SCE $N$ times to obtain $\|\mtx{H}_n^{(k)}-\mtx{D}_{\mu}\mtx{P}\|_{\mathrm{F}}^2$ and $\|\mtx{R}_n^{(k)}-\mtx{D}_{\mu}\mtx{P}\|_{\mathrm{F}}^2$, $k=1,2,\dots N$, and then average over $k$ to get the mean square error (MSE).
Figure~\ref{fig:dependence on n} and Figure~\ref{fig:dependence on eta} respectively show the linear dependence of 1/MSE on $\eta$ and $n$, while Figure~\ref{fig:dependence on d} demonstrates that the MSEs of two methods are both independent of the size of the state space. We can also find out that $\text{MSE}_{\text{MLE}}/\text{MSE}_{\text{SCE}}\approx 1/2$ when the spectral gap of $\mtx{P}$ is small, which supports the theoretical results presented in Theorem~\ref{thm:main results of natural counting} and Theorem~\ref{thm:main results of symmetric counting}.

\begin{figure}[!ht]
    \centering
        \begin{subfigure}[b]{0.32\textwidth}
            \includegraphics[width=1\textwidth]{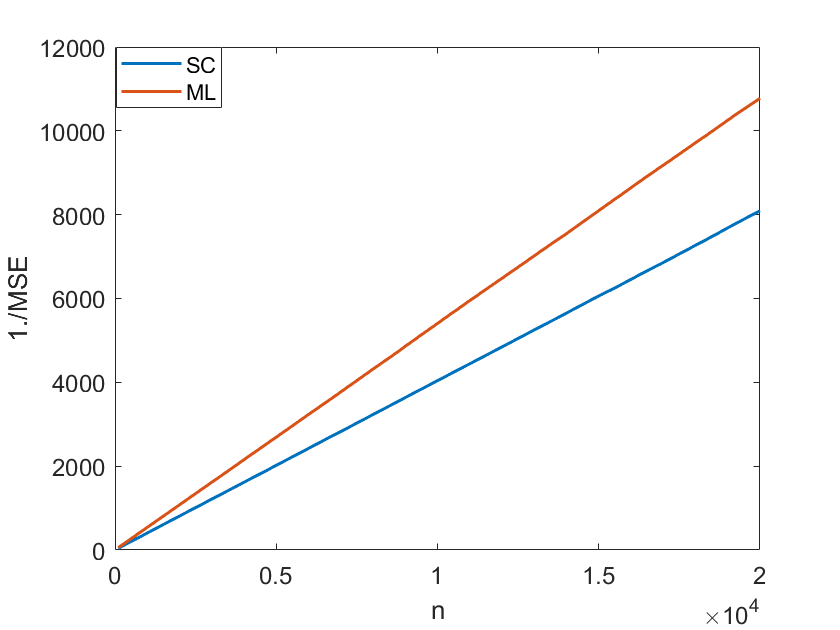}
            \caption{$\eta(\mtx{P})=0.5$}
        \end{subfigure}
        \begin{subfigure}[b]{0.32\textwidth}
            \includegraphics[width=1\textwidth]{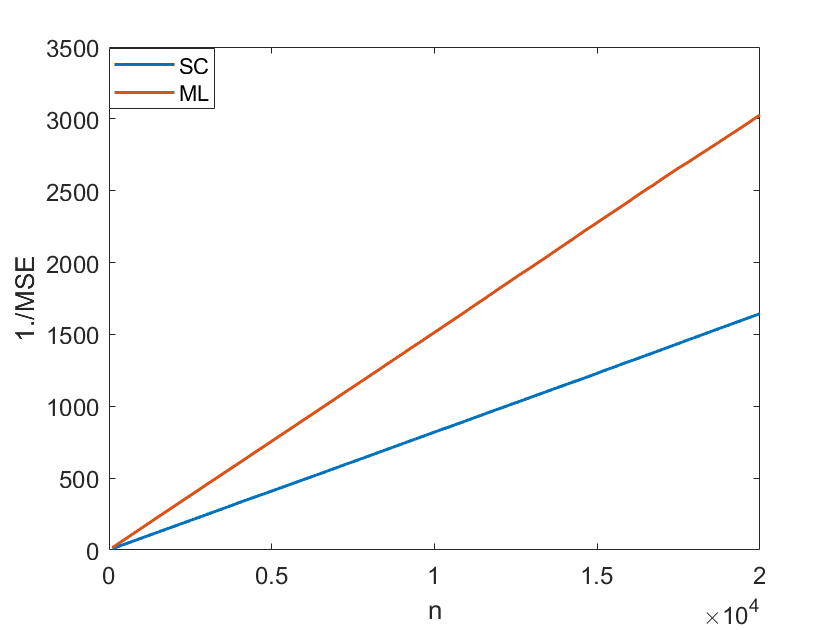}
            \caption{$\eta(\mtx{P})=0.2$}
        \end{subfigure}
        \begin{subfigure}[b]{0.32\textwidth}
            \includegraphics[width=1\textwidth]{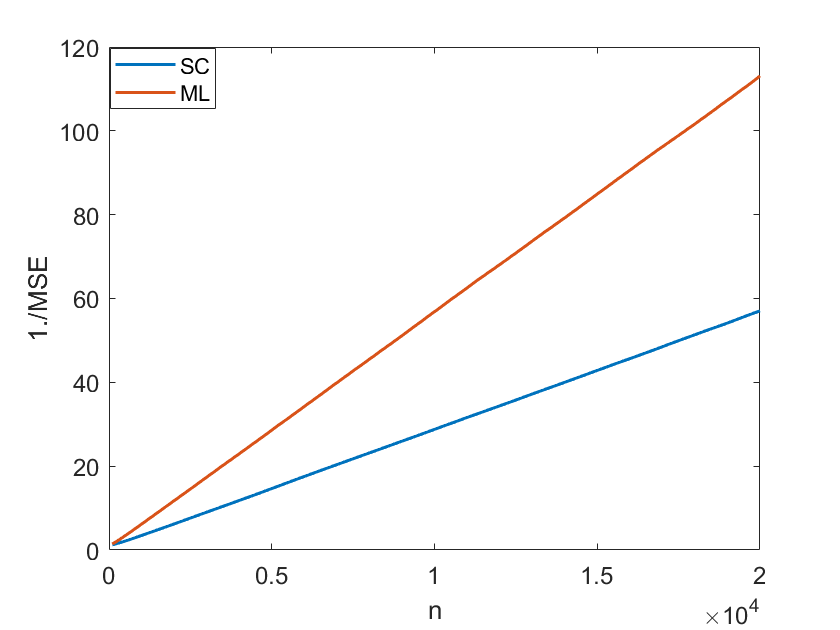}
            \caption{$\eta(\mtx{P})=0.01$}
        \end{subfigure}
        \caption{Plots of $1/\text{MSE}$ as the function of $n$ for $\eta(\mtx{P})=0.5,0.2,0.01$, respectively. We fix $N=10000$ and $d=50$.}
        \label{fig:dependence on n}
    \end{figure}

\begin{figure}[!ht]
    \centering
        \begin{subfigure}[b]{0.32\textwidth}
            \includegraphics[width=1\textwidth]{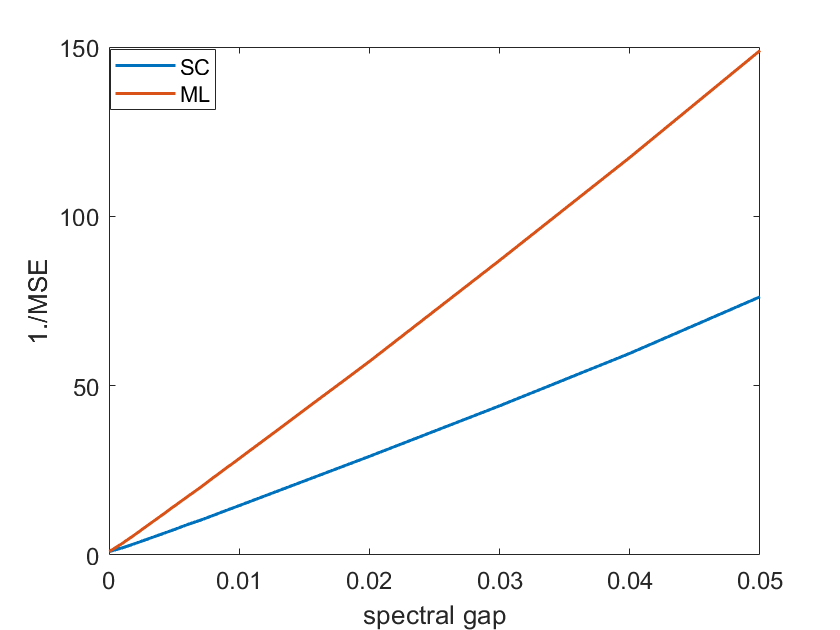}
            \caption{$n=5000$}
        \end{subfigure}
        \begin{subfigure}[b]{0.32\textwidth}
            \includegraphics[width=1\textwidth]{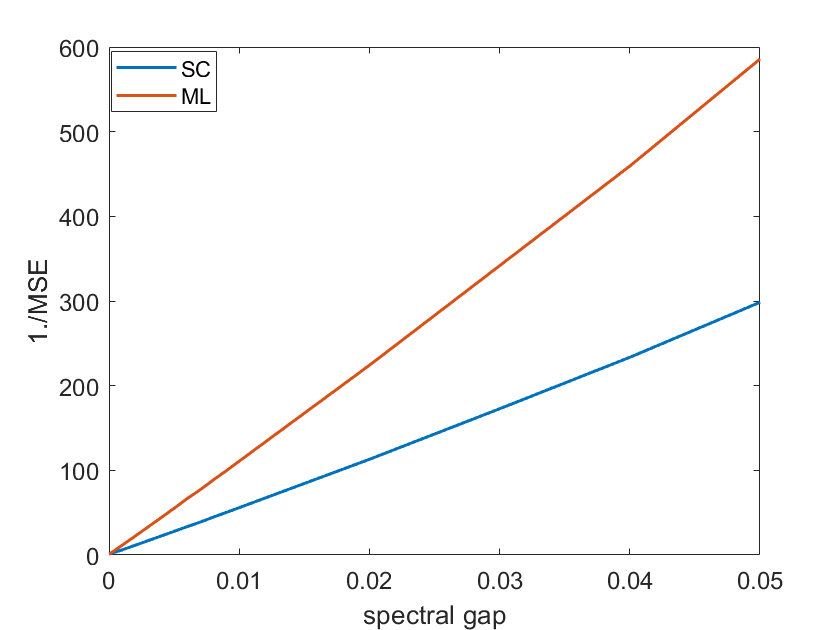}
            \caption{$n=20000$}
        \end{subfigure}
        \begin{subfigure}[b]{0.32\textwidth}
            \includegraphics[width=1\textwidth]{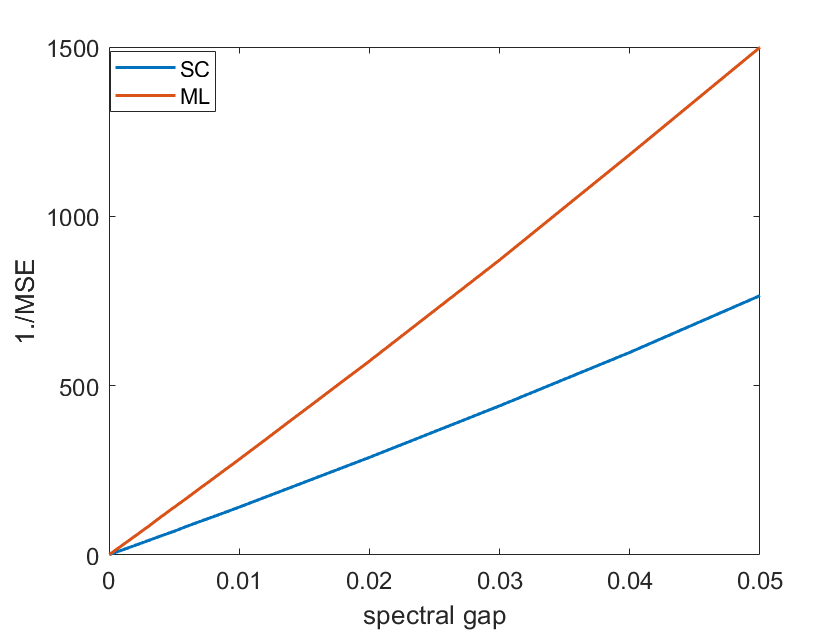}
            \caption{$n=50000$}
        \end{subfigure}
        \caption{Plots of $1/\text{MSE}$ as the function of $\eta(\mtx{P})$ for $n=5000,20000,50000$, respectively. We fix $N=10000$ and $d=50$.}
        \label{fig:dependence on eta}
    \end{figure}

    \begin{figure}[!ht]
        \centering
        \includegraphics[width=0.5\textwidth]{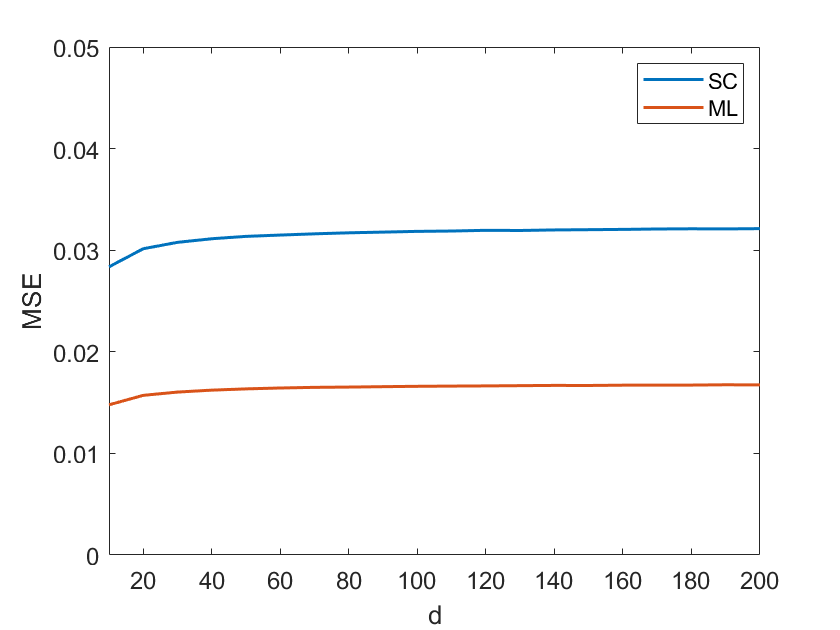}
        \caption{Plot of MSE as the function of $d$ for $N=5000$, $n=1000$. For each $d\in\left\{10,20,\ldots, 200\right\}$, we take the average of the MSEs of 100 different $d\times d$ transition matrices with fixed spectral gap $\eta=0.1$.}
        \label{fig:dependence on d}
        \end{figure}

\appendix
\section{Comparison of different spectral gaps}\label{sec:appendix}
 This section provides a brief comparison of different spectral gaps introduced in Section \ref{sec:main results}. Since all of these results already appear in \cite{chatterjee2025spectral,huang2026bernstein}, we present them here solely for the convenience of the reader. We first show that $\eta_p(\mtx{P})\geq \eta_s(\mtx{P})\geq \eta_a(\mtx{P})$ and $\eta_p(\mtx{P})\geq \eta_{ps}(\mtx{P})/2$. As a direct consequence, our main results also hold if $\mtx{P}$ admits a nonzero absolute spectral gap, a nonzero pseudo spectral gap, or a nonzero symmetric spectral gap.
\begin{lemma}\label{lem:comparison_of_spectral gap}
    It holds that $\eta_p(\mtx{P})\geq \eta_s(\mtx{P})\geq \eta_a(\mtx{P})$ and $\eta_p(\mtx{P})\geq \eta_{ps}(\mtx{P})/2$ for any irreducible transition matrix $\mtx{P}$ with an invariant distribution $\mu$. 
\end{lemma}
\begin{proof}
    For any $h\in L_{2,\mu}^0$,  by the Cauchy--Schwarz inequality,
    \[\frac{\la h,(\mtx{\Id}-\mtx{P})h\ra_{\mu}}{\|h\|_{\mu}^2}\leq \frac{\|(\mtx{\Id}-\mtx{P})h\|_{\mu}}{\|h\|_{\mu}}.\]
    Taking the infimum over $h \in L_{2,\mu}^0$, we deduce $\eta_p(\mtx{P})\geq \eta_s(\mtx{P})$. Similarly, since
    \[\frac{\langle h, \mtx{P}h \rangle_{\mu}}{ \|h\|_{\mu}^2}\leq \frac{\|\mtx{P}h\|_{\mu}}{\|h\|_{\mu}}\]
    holds for any $h \in L_{2,\mu}^0$, we deduce $\lambda_s(\mtx{P})\leq \lambda_a(\mtx{P})$, and therefore $\eta_s(\mtx{P})\geq \eta_a(\mtx{P})$.

    To prove $\eta_p(\mtx{P})\geq\eta_{ps}(\mtx{P})/2$, we first note that for any $k\in \N_+$, 
    \[\left\|\left(\mtx{P}^{k-1}-\mtx{P}^k\right)h\right\|_{\mu}=\left\|\mtx{P}^{k-1}(\mtx{\Id}-\mtx{P})h\right\|_{\mu}\leq \|(\mtx{\Id}-\mtx{P})h\|_{\mu}.\] Then, for any $h \in L_{2,\mu}^0$ and any $k\in \N_+$, we can compute by triangle inequality and telescoping that 
    \[ 1-\frac{\left\|\mtx{P}^k h\right\|_{\mu}}{\|h\|_{\mu}}\leq \frac{\left\|\left(\mtx{\Id}-\mtx{P}^k\right) h\right\|_{\mu}}{\|h\|_{\mu}}\leq \frac{1}{\|h\|_{\mu}} \sum\limits_{i=1}^k \left\|\left(\mtx{P}^{i-1}-\mtx{P}^i\right)h \right\|_{\mu}\leq k \frac{\|(\mtx{\Id}-\mtx{P})h\|_{\mu}}{\|h\|_{\mu}}.\]
    It follows by taking the infimum over all $h\in L_{2,\mu}^0$ that $1-\lambda_a(\mtx{P}^k)\leq k\eta_p(\mtx{P})$, and thus
    \[\eta_{ps}(\mtx{P})=\sup_{k\in \N_+}\frac{1-\lambda_a^2\left(\mtx{P}^k\right)}{k}\leq \sup_{k\in \N_+}\frac{2\left(1-\lambda_a\left(\mtx{P}^k\right)\right)}{k}\leq 2\eta_p(\mtx{P}).\]
    We have used above the fact that $\lambda_a(\mtx{P}^k)\leq 1$ \eqref{lambda_a leq 1}.
\end{proof}

Recall that the irreducible transition matrix $\mtx{P}_2$ constructed in Section \ref{sec:proofs} is already an example whose IP gap is positive while whose absolute spectral gap is zero. Below, we provide two more examples demonstrating the superiority of the IP gap.
\begin{example}  
    Chatterjee provided in \cite{chatterjee2025spectral} an example demonstrating that the ratio $\eta_a({\mtx{P}})/\eta_p(\mtx{P})$ can be arbitrarily small. Consider the one-dimensional random walk on the periodic grid $\Z/n\Z$ where the
    probability of not moving and that of moving one step to the right are both $1/2$. It is then not hard to check that the IP gap of this Markov chain is of order $1/n$, while the symmetric spectral gap and the absolute spectral gap are both of order $1/n^2$.
\end{example}
\begin{example}[Perturbation of a periodic transition matrix]
    Consider a parameterized family of transition matrices, 
    \[\mtx{P}_0=\begin{bmatrix}
                0& 1 &0  \\
                0& 0 &1 \\
                1& 0 &0  \\
        \end{bmatrix}, \quad \text{and }\quad \mtx{P}_{\varepsilon}=\varepsilon \mtx{\Id}+(1-\varepsilon)\mtx{P}_0, \quad \varepsilon\in [0,1].\]
   Suppose $\varepsilon \ll 1$. One can easily show by a straightforward calculation and the spectral perturbation theory that $\eta_p(\mtx{P}_{\varepsilon})=(1-\varepsilon)\sqrt{3}$, $\eta_s(\mtx{P}_{\varepsilon})=3(1-\varepsilon)/2$, $\eta_{ps}(\mtx{P}_{\varepsilon})=\mathcal{O}(\varepsilon)$, $\eta_{a}(\mtx{P}_{\varepsilon})=\mathcal{O}(\varepsilon)$. In particular, $\eta_p(\mtx{P}_0)=\sqrt{3}$, $\eta_s(\mtx{P}_0)=3/2$ and $\eta_a(\mtx{P}_0)=\eta_{ps}(\mtx{P}_0)=0$.
\end{example}

\subsection*{Acknowledgements.} The authors are supported by the National Key R\&D Program of China under the grant 2021YFA1001500.\\



\bibliographystyle{myalpha}
\bibliography{reference}

\end{document}

%% file: macro-3.tex
\usepackage[usenames,dvipsnames]{xcolor}
\usepackage{fancyhdr}
\usepackage{amsmath,amsfonts,amsbsy,amsgen,amscd,amssymb,amsthm}
\usepackage{url}

\usepackage{mathtools}

\usepackage[font=small,margin=0.25in,labelfont={bf},labelsep={colon}]{caption}

\usepackage{subcaption}

\usepackage{tikz}
\usepackage{microtype}
\usepackage{enumitem}

\definecolor{dark-gray}{gray}{0.3}
\definecolor{dkgray}{rgb}{.4,.4,.4}
\definecolor{dkblue}{rgb}{0,0,.5}
\definecolor{medblue}{rgb}{0,0,.75}
\definecolor{rust}{rgb}{0.5,0.1,0.1}

\usepackage[colorlinks=true]{hyperref}

\hypersetup{urlcolor=rust}
\hypersetup{citecolor=dkblue}
\hypersetup{linkcolor=dkblue}

\usepackage{setspace}

\usepackage{graphicx}
\usepackage{booktabs,longtable,tabu} 
\usepackage{multirow} 

\usepackage{float}

\usepackage[full]{textcomp}

\usepackage[scaled=.98,sups]{XCharter}
\usepackage[scaled=1.04,varqu,varl]{inconsolata}
\usepackage[type1]{cabin}
\usepackage[charter,vvarbb,scaled=1.07]{newtxmath}
\usepackage[cal=boondoxo]{mathalfa}
\usepackage[T1]{fontenc}

\usepackage{bm} 

\graphicspath{{figures/}}

\newtheorem{theorem}{Theorem}[section]
\newtheorem{lemma}[theorem]{Lemma}

\theoremstyle{definition}

\newtheorem{example}[theorem]{Example}
\newtheorem{remark}[theorem]{Remark}

\numberwithin{equation}{section} 
\numberwithin{figure}{section}
\numberwithin{table}{section}

\floatstyle{plaintop}
\newfloat{recipe}{thp}{lor}
\floatname{recipe}{Recipe}
\numberwithin{recipe}{section}

\providecommand{\mathbold}[1]{\bm{#1}}  
                                
\renewcommand{\phi}{\varphi}

\newcommand{\econst}{\mathrm{e}}

\newcommand{\Id}{\mathbf{I}}

\providecommand{\mathbbm}{\mathbb} 

\newcommand{\R}{\mathbbm{R}}

\newcommand{\N}{\mathbbm{N}}
\newcommand{\Z}{\mathbbm{Z}}

\newcommand{\diff}[1]{\mathrm{d}{#1}}

\newcommand{\Prob}[1]{\mathbbm{P}\left\{{#1}\right\}}

\newcommand{\Expect}{\operatorname{\mathbbm{E}}}

\DeclareMathOperator{\Var}{Var}

\newcommand{\mtx}[1]{\mathbold{#1}}

\newcommand{\trace}{\operatorname{tr}}

\newcommand{\triplenorm}[1]{{\left\vert\kern-0.25ex\left\vert\kern-0.25ex\left\vert #1
    \right\vert\kern-0.25ex\right\vert\kern-0.25ex\right\vert}}

\newcommand{\E}{\Expect}

